\documentclass[12pt,a4paper]{amsart}
\usepackage{amsmath,amstext,amssymb,amsthm,amsfonts}
\newcommand{\nc}{\newcommand}

\nc{\one}{\mbox{\bf 1}}
\nc{\invtensor}{\underset{\leftarrow}{\otimes}}
\nc{\const}{\operatorname{const}}

\nc{\ad}{\operatorname{ad}}
\nc{\tr}{\operatorname{tr}}

\nc{\tp}{\operatorname{top}}
\nc{\rank}{\operatorname{rank}}
\nc{\corank}{\operatorname{corank}}
\nc{\codim}{\operatorname{codim}}
\nc{\sdim}{\operatorname{sdim}}
\nc{\mult}{\operatorname{mult}}

\nc{\spn}{\operatorname{span}}
\nc{\Sym}{\operatorname{Sym}}
\nc{\sym}{\operatorname{sym}}
\nc{\id}{\operatorname{id}}
\nc{\Id}{\operatorname{Id}}
\nc{\Ree}{\operatorname{Re}}

\nc{\htt}{\operatorname{ht}}
\nc{\sch}{\operatorname{sch}}
\nc{\str}{\operatorname{str}}

\nc{\Ker}{\operatorname{Ker}}
\nc{\rker}{\operatorname{rKer}}
\nc{\im}{\operatorname{Im}}
\nc{\osp}{\mathfrak{osp}}
\nc{\sgn}{\operatorname{sgn}}
\nc{\F}{\operatorname{F}}
\nc{\Mod}{\operatorname{Mod}}
\nc{\Mat}{\operatorname{Mat}}
\nc{\Soc}{\operatorname{Soc}}
\nc{\Inj}{\operatorname{Inj}}
\nc{\Hom}{\operatorname{Hom}}
\nc{\End}{\operatorname{End}}
\nc{\supp}{\operatorname{supp}}
\nc{\Card}{\operatorname{Card}}
\nc{\Ann}{\operatorname{Ann}}
\nc{\Ind}{\operatorname{Ind}}
\nc{\Coind}{\operatorname{Coind}}
\nc{\wt}{\operatorname{wt}}
\nc{\ch}{\operatorname{ch}}
\nc{\Stab}{\operatorname{Stab}}
\nc{\Sch}{{\mathcal S}\mbox{\em ch}}
\nc{\Irr}{\operatorname{Irr}}
\nc{\Spec}{\operatorname{Spec}}
\nc{\Prim}{\operatorname{Prim}}
\nc{\Aut}{\operatorname{Aut}}
\nc{\Ext}{\operatorname{Ext}}

\nc{\Fract}{\operatorname{Fract}}
\nc{\gr}{\operatorname{gr}}
\nc{\deff}{\operatorname{def}}
\nc{\HC}{\operatorname{HC}}

\nc{\red}{\operatorname{red}}

\nc{\wdchi}{\widetilde{\chi}}
\nc{\wdH}{\widetilde{H}}
\nc{\wdN}{\widetilde{N}}
\nc{\wdM}{\widetilde{M}}
\nc{\wdO}{\widetilde{O}}
\nc{\wdR}{\widetilde{R}}
\nc{\wdS}{\widetilde{S}}
\nc{\wdV}{\widetilde{V}}

\nc{\wdC}{\widetilde{C}}

\nc{\Obj}{\operatorname{Obj}}
\nc{\Dglie}{\operatorname{{\mathcal D}glie}}
\nc{\Fin}{\operatorname{{\mathcal F}in}}

\nc{\Adm}{\operatorname{\mathcal{A}dm}}

\nc{\Sg}{{\cS(\fg)}}
\nc{\Shg}{{\cS(\fhg)}}
\nc{\Ug}{{\cU(\fg)}}
\nc{\Uhg}{{\cU(\fhg)}}
\nc{\Sh}{{\cS(\fh)}}
\nc{\Uh}{{\cU(\fh)}}
\nc{\Uhh}{{\cU(\fhh)}}
\nc{\Zg}{{{\mathcal{Z}}(\fg)}}

\nc{\Vir}{{\mathcal{V}ir}}
\nc{\NS}{{\mathcal{N}S}}

\nc{\tZg}{{\widetilde{\mathcal Z}({\mathfrak g})}}
\nc{\Zk}{{\mathcal Z}({\mathfrak k})}

\nc{\Up}{{\mathcal U}({\mathfrak p})}
\nc{\Ah}{{\mathcal A}({\mathfrak h})}
\nc{\Ag}{{\mathcal A}({\mathfrak g})}
\nc{\Ap}{{\mathcal A}({\mathfrak p})}
\nc{\Zp}{{\mathcal Z}({\mathfrak p})}
\nc{\cR}{\mathcal R}
\nc{\cS}{\mathcal S}
\nc{\cT}{\mathcal{T}}
\nc{\cY}{\mathcal Y}
\nc{\cA}{\mathcal A}
\nc{\cU}{\mathcal U}
\nc{\cH}{\mathcal H}
\nc{\cM}{\mathcal M}
\nc{\cL}{\mathcal L}
\nc{\cF}{\mathcal F}
\nc{\fg}{\mathfrak g}

\nc{\fo}{\mathfrak o}

\nc{\CO}{\mathcal O}
\nc{\CR}{\mathcal R}
\nc{\Cl}{\mathcal {C}\ell}

\nc{\cW}{\mathcal{W}}
\nc{\bM}{\mathbf{M}}
\nc{\bL}{\mathbf{L}}
\nc{\bN}{\mathbf{N}}

\nc{\zq}{\mathpzc q}

\nc{\fl}{\mathfrak l}
\nc{\fn}{\mathfrak n}
\nc{\fm}{\mathfrak m}
\nc{\fp}{\mathfrak p}
\nc{\fh}{\mathfrak h}
\nc{\ft}{\mathfrak t}
\nc{\fk}{\mathfrak k}
\nc{\fb}{\mathfrak b}
\nc{\fs}{\mathfrak s}

\nc{\psl}{\mathfrak{psl}}
\nc{\fB}{\mathfrak B}

\nc{\vareps}{\varepsilon}
\nc{\varesp}{\varepsilon}
\nc{\veps}{\varepsilon}

\nc{\fsl}{\mathfrak{sl}}
\nc{\fgl}{\mathfrak{gl}}
\nc{\fso}{\mathfrak{so}}

\nc{\fpq}{\mathfrak{pq}}
\nc{\fq}{\mathfrak q}
\nc{\fsq}{\mathfrak{sq}}
\nc{\fpsq}{\mathfrak{psq}}

\nc{\fhg}{\hat{\fg}}
\nc{\fhn}{\hat{\fn}}
\nc{\fhh}{\hat{\fh}}
\nc{\fhb}{\hat{\fb}}
\nc{\hrho}{\hat{\rho}}

\nc{\hsl}{\hat{\fsl}}
\nc{\fpo}{\mathfrak{po}}
\nc{\dirlim}{\underset{\rightarrow}{\lim}\,}
\nc{\nen}{\newenvironment}
\nc{\ol}{\overline}
\nc{\ul}{\underline}
\nc{\ra}{\rightarrow}
\nc{\lra}{\longrightarrow}
\nc{\Lra}{\Longrightarrow}
\nc{\bo}{\bar{1}}
\nc{\Lla}{\Longleftarrow}

\nc{\Llra}{\Longleftrightarrow}

\nc{\thla}{\twoheadleftarrow}

\nc{\lang}{(}
\nc{\rang}{)}

\nc{\hra}{\hookrightarrow}

\nc{\iso}{\overset{\sim}{\lra}}

\nc{\ssubset}{\underset{\not=}{\subset}}

\nc{\vac}{|0\rang}

\nc{\Thm}[1]{Theorem~\ref{#1}}
\nc{\Prop}[1]{Proposition~\ref{#1}}
\nc{\Lem}[1]{Lemma~\ref{#1}}
\nc{\Cor}[1]{Corollary~\ref{#1}}
\nc{\Conj}[1]{Conjecture~\ref{#1}}
\nc{\Claim}[1]{Claim~\ref{#1}}
\nc{\Defn}[1]{Definition~\ref{#1}}
\nc{\Exa}[1]{Example~\ref{#1}}
\nc{\Rem}[1]{Remark~\ref{#1}}
\nc{\Note}[1]{Note~\ref{#1}}
\nc{\Quest}[1]{Question~\ref{#1}}
\nc{\Hyp}[1]{Hypoth\`ese~\ref{#1}}
\nen{thm}[1]{\label{#1}{\bf Theorem.\ } \em}{}
\nen{prop}[1]{\label{#1}{\bf Proposition.\ } \em}{}
\nen{lem}[1]{\label{#1}{\bf Lemma.\ } \em}{}
\nen{cor}[1]{\label{#1}{\bf Corollary.\ } \em}{}
\nen{conj}[1]{\label{#1}{\bf Conjecture.\ } \em}{}

\nen{claim}[1]{\label{#1}{\bf Claim.\ } \em}{}

\nen{defn}[1]{\label{#1}{\bf Definition.\ } }{}
\nen{exa}[1]{\label{#1}{\bf Example.\ } }{}
\nen{rem}[1]{\label{#1}{\em Remark.\ } }{}
\nen{note}[1]{\label{#1}{\em Note.\ } }{}
\nen{exer}[1]{\label{#1}{\em Exercise.\ } }{}
\nen{sket}[1]{\label{#1}{\em Sketch of proof.\ } }{}
\nen{quest}[1]{\label{#1}{\bf Question.\ } \em}{}

\nen{hyp}[1]{\label{#1}{\bf Hypoth\`ese.\ } \em}{}
\begin{document}

\setcounter{section}{-1}
\setcounter{tocdepth}{1}

\title[Generalized Reflection Root Systems ]
{Generalized Reflection Root Systems}

\author[Maria Gorelik, Ary Shaviv]{Maria Gorelik$^\dag$, Ary Shaviv }

\address[]{Dept. of Mathematics, The Weizmann Institute of Science,
Rehovot 76100, Israel}
\email{maria.gorelik@weizmann.ac.il, ary.shaviv@weizmann.ac.il}
\thanks{$^\dag$ Supported in part by BSF Grant No. 711623.}

\maketitle

\begin{abstract}
We study a combinatorial object, which we call a
GRRS (generalized reflection root system);
the classical root systems and GRSs introduced by V.~Serganova
are examples of finite GRRSs. A GRRS is finite if it contains
a finite number of vectors and is called affine if it is infinite and has a finite minimal quotient. We prove that
an irreducible GRRS containing an isotropic root is either finite or affine; we describe all finite and affine GRRSs
and classify them in most of the cases.

\end{abstract}

\section{Introduction}
We study a combinatorial object, which we call a
GRRS (generalized reflection root system), see~\Defn{defnGRRS}.
The classical root systems are finite GRRSs without isotropic roots. Our definition of GRRS is motivated by Serganova's definition
of GRS introduced in~\cite{VGRS}, Sect. 1, and by
the following examples:
the set of real roots $\Delta_{re}$ of a symmetrizable
Kac-Moody superalgebra introduced in~\cite{S2} and its subsets  $\Delta_{re}(\lambda)$  ("integral real roots"),
see~\cite{GK}.

Each GRRS $R$ is, by definition,  a subset of a finite-dimensional complex
vector space $V$ endowed with a symmetric bilinear form $(-,-)$.
The image of $R$ in $V/Ker (-,-)$ is denoted by $cl(R)$; it satisfies weaker
properties than GRRS and is called a WGRS.
An infinite GRRS is called {\em affine} if its image
$cl(R)$ is finite (in this case $cl(R)$ is a finite WGRS, which
were classified in~\cite{VGRS}).
 We show that an irreducible GRRS
 containing an isotropic root is either finite or affine.
 Recall a theorem of C.~Hoyt that a symmetrizable
 Kac-Moody superalgebra with an isotropic simple root
 and an indecomposable Cartan matrix (this corresponds to
 the irreducibility of GRRS)
 is finite-dimensional or affine, see~\cite{H}.

Finite  GRRSs correspond to the root systems of finite-dimensional Kac-Moody superalgebras.
In this paper we describe all affine GRRSs $R$ and classify them
for most cases of $cl(R)$.  Irreducible affine
GRRSs
with $\dim Ker (-,-)=1$   correspond to symmetrizable affine Lie superalgebras. This case
 was treated in~\cite{Sh}; in particular,
it implies that an "irreducible subsystem" of the set of real roots of an affine Kac-Moody superalgebra is a set of real roots
of an affine or a finite-dimensional Kac-Moody superalgebra
(this  was used in~\cite{GK}).

For each GRRS we introduce a certain subgroup of $Aut R$, which is denoted
by $GW(R)$; if $R$ does not contain isotropic roots, then $GW(R)$ is the usual Weyl group.  Let $R$ be an irreducible
affine GRRS (i.e., $cl(R)$ is finite). We show
that if the action of
$GW(cl(R))$ to $cl(R)$ is transitive and $cl(R)\not=A_1$,
then $R$ is either the affinization of $cl(R)$
(see~\S~\ref{defaff} for definition) or, if $cl(R)$ is the root system
of $\mathfrak{psl}(n,n), n>2$, $R$ is
a certain ``bijective quotient'' of the affinization
of the root system of $\mathfrak{pgl}(n,n)$, see~\S~\ref{bijquo} for definition. The action of
$GW(cl(R))$ to $cl(R)$ is transitive if and only if
$cl(R)$ is the root system of a simply laced
Lie algebra or a Lie superalgebra $\fg\not=B(m,n)$, which is not a Lie algebra. If $R$ is such that $cl(R)=B(m,n), m,n\geq 1$ or
$cl(R)=B_n,C_n, n\geq 3$, then
 $R$ is classified by non-empty subsets of the affine
space $\mathbb{F}_2^k$ up to affine autormorphisms of $\mathbb{F}_2^k$, where $\dim Ker(-,-)=k$;
a similar classification holds for $cl(R)=A_1$.
In the  cases $cl(R)=G_2, F_4$, the GRRSs $R$
are parametrized by $s=0,1,\ldots,\dim Ker(-,-)$.
In the remaining case either $cl(R)$ is a finite WGRS, which is not a GRRS, or $cl(R)=BC_n$;  we partially classify
the corresponding GRRSs (we describe all possible $R$).

Another combinatorial object, an extended affine root supersystem (EARS),
was introduced and described in a recent paper of
M.~Yousofzadeh~\cite{You}.  The main differences between
a GRRS and an EARS are the following: EARS has a "string property"
(for each $\alpha,\beta$ in an EARS with $(\alpha,\alpha)\not=0$
the intersection of $\beta+\mathbb{Z}\alpha$
with the EARS is a string $\{\beta-j\alpha|\ j\in \{-p,p+1,\ldots,q\}\}$ for some $p,q\in\mathbb{Z}$ with
$p-q=2(\alpha,\beta)/(\alpha,\alpha)$), and  a GRRS
should be invariant with respect to the "reflections"
connected to its elements. The string property implies the
invariance with respect to the reflections connected to
non-isotropic roots ($\alpha$ such that $(\alpha,\alpha)\not=0$).
A finite GRRS corresponds to the root system
of a finite-dimensional Kac-Moody superalgebra, and the
finite EARSs include two additional series.
The root system of a symmetrizable affine
Lie superalgebra is an EARS and the set of real roots is a GRRS.
Moreover, the set of roots of a symmetrizable Kac-Moody superalgebra is an EARS only if this algebra is affine or finite-dimensional
(by contrast, the set of real roots is always a GRRS).
For example, the real roots of a Kac-Moody algebra with the Cartan matrix $\begin{pmatrix} 2 &-3\\-3 &2 \end{pmatrix}$ form a GRRS, which can not be embedded in an EARS. However, according to theorem of C.~Hoyt~\cite{H}, an indecomposable symmetrizable Kac-Moody superalgebra with an isotropic real root is affine, so
 there are no examples of this nature if the GRRS contains an isotropic root.  Eventhough the GRRSs are not exhausted
 by GRRSs coming from Kac-Moody algebra, from  Prop. 3.2
 in~\cite{You} it follows that
 an affine GRRS $R$ can be always embedded in an EARS, i.e.
 there exists an EARS $R'$ such that
 $R=\{\alpha\in R'|\ \exists\beta\in R'\ (\alpha,\beta)\not=0\}$.
  This allows to obtain a description of affine GRRSs  from the description of EARS in~\cite{You}, \cite{Yos}
  and, using~\Thm{thmdirsum}, to obtain a description of
  the irreducible GRRSs containing isotropic roots.

In Section~\ref{sect1} we give all definitions, examples of GRRSs
and explain the connection between GRRS, GRS introduced in~\cite{VGRS} and root systems of Kac-Moody superalgebras.

In Section~\ref{sectnondeg} we prove that if $R$ is an irreducible GRRS with a non-degenerate symmetric bilinear form and $R$ contains an isotropic root, then $cl(R)$ is finite (and is classified in~\cite{VGRS}).

In Section~\ref{sect3} we prove some lemmas, which are used
later.

In Section~\ref{sect4} we obtain a classification of $R$
for the case when $cl(R)$ is finite and is generated by a basis of $cl(V)$.

In Section~\ref{Ann} we obtain a classification of $R$
for the case when $cl(R)$ is the roots system of
$\mathfrak{psl}(n+1,n+1), n>1$. This is the only situation when
 the form $(-,-)$ is degenerate and
 $R$ can be finite; this holds in the case $\mathfrak{gl}(n,n)$.

In Section~\ref{sect6} we obtain a classification of $R$
for the case when $cl(R)$ is a finite WGRS, which is not a GRS
($cl(R)=BC(m,n), C(m,n)$) and describe $R$ for the remaining case
$cl(R)=BC_n$. This completes the description of GRRSs $R$ with finite $cl(R)$.

In Section~\ref{sect7} we present the correspondence
between the irreducible affine GRRSs
with $\dim Ker (-,-)=1$  and the symmetrizable affine Lie superalgebras.

\section{Definitions and basic examples}\label{sect1}
In this section we introduce the notion GRRS (generalized reflection root systems) and consider several examples.

\subsection{Notation}
 Throughout the paper $V$ will be a finite-dimensional complex vector space with a symmetric bilinear form $(-,-)$.

For  $\alpha\in V$ with $(\alpha,\alpha)\not=0$
we introduce a notation
$$k_{\alpha,\beta}:=\frac{2(\alpha,\beta)}{(\alpha,\alpha)}$$
for each $\beta\in V$,
 and we define the reflection
$r_{\alpha}\in \End V$ by the usual formula
$$r_{\alpha}(v):=v-k_{\alpha,v} v.$$
Clearly, $r_{\alpha}$ preserves $(-,-)$. Note that
\begin{equation}\label{kkk}
k_{\alpha,r_{\gamma}\beta}=k_{\alpha,\beta}-k_{\alpha, \gamma}
k_{\gamma,\beta}
\end{equation}
if $(\alpha,\alpha), (\gamma,\gamma)\not=0$.

We use the following notation: if $X$ is a subset of $V$, then
 $X^{\perp}:=\{v\in V|\ \forall x\in X\ (x,v)=0\}$ and
 $\mathbb{Z}X$ is the additive subgroup of $V$ generated by $X$
 (similarly, $\mathbb{C}X$ is a subspace of $V$ generated by $X$).

\subsection{}\label{defGRRS}
\begin{defn}{defnGRRS}
Let $V$ be  a finite-dimensional complex vector space with a symmetric bilinear form $(-,-)$.
A non-empty set $R\subset V$ is called a {\em generalized reflection root system (GRRS)}  if the following axioms hold

(GR0)   $\Ker (-,-)\cap R=\emptyset$;

(GR1) the canonical map $\mathbb{Z}R\otimes_{\mathbb{Z}}\mathbb{C}\to V$ is a bijection;

(GR2)  for each $\alpha\in R$ with $(\alpha,\alpha)\not=0$
one has $r_{\alpha}R=R$; moreover, $\beta-r_{\alpha}\beta\in\mathbb{Z}\alpha$
for each $\beta\in R$;

(GR3) for each $\alpha\in R$  with $(\alpha,\alpha)=0$ there exists
an invertible map $r_{\alpha}: R\to R$ such that
\begin{equation}\label{riso}
\begin{array}{l}
r_{\alpha}(\alpha)=-\alpha,\ \ r_{\alpha}(-\alpha)=\alpha,\\
r_{\alpha}(\beta)=\beta\ \text{ if }\beta\not=\pm\alpha,\ (\alpha,\beta)=0,\\
r_{\alpha}(\beta)\in\{\beta\pm\alpha\}\ \text{ if }(\alpha,\beta)\not=0.
\end{array}\end{equation}
\end{defn}

\subsubsection{}\label{nullity}
We sometimes write $R\subset V$ is a GRRS, meaning that
$R$ is a GRRS in $V$. If $R\subset V$ is a GRRS, we call $\alpha\in R$ a {\em root};
we call a root $\alpha$ {\em isotropic} if $(\alpha,\alpha)=0$.

\subsubsection{}
\begin{defn}{}
We call a GRRS $R\subset V$ {\em affine}
if $R$ is infinite and the image of $R$ in $V/Ker (-,-)$ is finite.
\end{defn}

\subsubsection{Remarks}\label{alphabeta}
Observe that $R=-R$ if $R$ is a GRRS.
By~\cite{VGRS}, Lem. 1.11, the axiom
(GR3) is equivalent to $R=-R$ and
the condition that for each $\alpha,\beta\in R$
with $(\alpha,\alpha)=0\not=(\alpha,\beta)$
the set $\{\beta\pm\alpha\}\cap R$
contains exactly one element.
In particular, if $R$ is a GRRS, then $r_{\beta}$ is an involution
 and it is uniquely defined for any $\beta\in R$.

In~\Thm{thmdirsum} we will show that if $(-,-)$ is non-degenerate, then
(GR1) is equivalent to the condition that $R$ spans $V$.

\subsubsection{Weyl group and $GW(R)$}\label{Weyl}
For any $X\subset V$ denote by $W(X)$ the group
generated by $\{r_{\alpha}|\ \alpha\in X, (\alpha,\alpha)\not=0\}$. Clearly, $W(R)$ preserves the bilinear form $(-,-)$.
If $R\subset V$ is a GRRS, we call $W(R)$
{\em the Weyl group} of $R$. By (GR2)
$R$ is $W(R)$-invariant.

If $R$ is a GRRS, then to each $\alpha\in R$ we assigned an involution
$r_{\alpha}\in Aut(R)$; we denote by $GW(R)$  the subgroup
of $Aut(R)$ generated by these involutions.

\subsubsection{}
In~\cite{VGRS}, Sect. 7,
 V.~Serganova considered another object,
where $r_{\alpha}$ were not assumed to be invertible, i.e.
(GR3) is substituted by

(WGR3)  for each $\alpha\in R$  with $(\alpha,\alpha)=0$ there exists
a map $r_{\alpha}: R\to R$  satisfying~(\ref{riso}).

If $V$ is endowed with a non-degenerate form and $R\subset V$
satisfies (GR0)-(GR2) and (WGR3), we
call $R$ a {\em weak GRS (WGRS)}; the finite WGRS were
classified in~\cite{VGRS}, Sect. 7.

\subsubsection{}\label{alphabeta}
Note that $R=-R$ if $R$ is a WGRS.
By~\cite{VGRS}, Lem. 1.11, the axiom
(WGR3) (resp., (GR3)) is equivalent to $R=-R$ and
for each isotropic $\alpha\in R$ the set $\{\beta\pm\alpha\}\cap R$
is non-empty (resp., contains exactly one element) if $\beta\in R$
is such that $(\beta,\alpha)\not=0$.

\subsection{Other definitions}
Classical root systems can be naturally viewed as examples of GRRS, see~\S~\ref{clRS}
below. The following definitions are motivated by this example.

\subsubsection{Subsystems}\label{subsystem}
For a GRRS $R\subset V$ we call $R'\subset R$
a {\em subsystem of $R$} if $R'$ is a GRRS in $\mathbb{C}R'$.

It turns out that $GW(R)$ does not preserve the subsystems:
$B_2$ can be naturally viewed as a subsystem of $B(2,1)$,
but $r_{\alpha}(B_2)$ is not a subsystem if $\alpha$ is isotropic.

If $R'\subset R$ does not contain isotropic roots, then $R'$ is a subsystem if and only if $R'$ is non-empty and $r_{\alpha}R'=R'$
for any $\alpha\in R'$ (note that if $\alpha$ is isotropic, then
$R':=\{\pm \alpha\}$ is not a GRRS, even though
$r_{\alpha}R'=R'$).

Note that for any non-empty $S\subset R$ the intersection $span S\cap R$
is a GRRS in  $span S$ if and only if (GR0) holds
(for any $\alpha\in (span S\cap R)$ there exists $\beta\in S$
such that $(\alpha,\beta)\not=0$).

 We say that  a non-empty set $X\subset R$ {\em generates a subsystem $R'\subset R$}
 if $R'$ is a unique minimal (by inclusion) subsystem
 containing $X$ (i.e.,
 for any subsystem $R''\subset R$ with $X\subset R''$
 one has $R'\subset R''$). In particular,
 $R$ is generated by $X$
 if $R$ is a minimal GRRS  containing $X$.

\subsubsection{}\label{reducible}
 We call a GRRS $R$ {\em reducible} if $R=R_1\cup R_2$,
 where $R_1, R_2$ are non-empty
and $(R_1,R_2)=0$. Note that in this case $R=R_1\coprod R_2$ and
$R_1, R_2$ are subsystems of $R$.
We call a GRRS $R$ {\em irreducible} if
 $R$ is not reducible.

If the bilinear form $(-,-)$ is non-degenerate on $V$,
then any GRRS is of the form $\coprod_{i=1}^k R_i\subset \oplus_{i=1}^k V_i$,
where $R_i\subset V_i$ is an irreducible GRRS.

\subsubsection{Isomorphisms}
\label{iso}
We say that two GRRSs $R\subset V,\ R'\subset V'$ are isomorphic
if there exists a linear homothety $\iota: V\to V'$ such that
$\iota(R)=R'$ (by a "homothety" we mean that $\iota$ is a
is linear isomorphism and there exists $x\in\mathbb{C}^*$
such that $(\iota(v), \iota(w))=x(v,w)$ for all $v,w\in V$).

 \subsubsection{Reduced GRRS}
 From (GR2), (GR3) one has $R=-R$.
 A GRRS $R$ is called {\em reduced} if  $\alpha,\lambda\alpha\in R$ for some $\lambda\in\mathbb{C}$ forces $\lambda=\pm 1$.
 (It is easy to see that this always holds if $\alpha$ is isotropic; if $\alpha$ is non-isotropic, then (GR2)
 gives $\lambda\in\{\pm 1,\pm \frac{1}{2},\pm 2\}$).

 \subsection{Examples}
 Let us consider several examples of GRRSs.

 \subsubsection{Classical root systems}\label{clRS}
Recall that a classical root system is a finite subset $R$
in a Euclidean space $V$ with the properties:
$0\not\in R$, $r_{\alpha}R=R$ for each $\alpha\in R$
and $r_{\alpha}\beta-\beta\in \mathbb{Z}\alpha$ for each
$\alpha,\beta\in R$. We see that $R$ is a finite GRRS in the complexification of $V$. Using~\cite{Ser}, Ch. V, it is easy to show that all finite GRRSs without
isotropic roots are of this form: $R\subset V$ is
a  finite GRRS without
isotropic roots if and only if $R\subset\mathbb{R}R$ is a classical root system.

The classical root systems were classified by  W.~Killing and E.~Cartan: the reduced irreducible classical root systems
are the series $A_n, n\geq 1$,
$B_n, n\geq 2$, $C_n, n\geq 3$, $D_n, n\geq 4$ and the exceptional root systems
$E_6, E_7, E_8, F_4, G_2$ (the lower index always stands for the dimension of $V$);
sometimes we use the notations $C_1:=A_1,C_2:=B_2$ and $D_3:=A_3$.
The irreducible non-reduced root systems of finite type are of the form $BC_n=B_n\cup C_n, n\geq 1$.

The reduced irreducible classical root systems
are the root systems of  finite-dimensional
simple complex Lie algebras.

\subsubsection{Example: GRSs introduced by V.~Serganova}
\label{GRS}
A GRS introduced by V.~Serganova in~\cite{VGRS}, Sect. 1 is
a finite GRRS $R\subset V$ with a non-degenerate form
$(-,-)$. V.~Serganova  classified these systems.
Recall the results of these classification.

A  complex simple finite-dimensional Lie superalgebras
$\fg=\fg_{\ol{0}}\oplus \fg_{\ol{1}}$ is called {\em basic classical} if
$\fg_{\ol{0}}$ is reductive and $\fg$
admits a non-degenerate invariant symmetric bilinear form $B$ with $B(\fg_{\ol{0}},\fg_{\ol{1}})=0$.
This bilinear form induces a non-degenerate
symmetric bilinear form on $\fh^*$, where $\fh$ is a Cartan subalgebra of $\fg$.
The set of roots of $\fg$ is a GRS  in $\fh^*$ if $\fg\not=\mathfrak{psl}(2,2)$.
Conversely, any GRS
is the root system of a basic classical Lie superalgebra
 differ from $\psl(2,2)$ (in particular, the non-reduced classical root system
 $BC_n$ is the root system of a basic classical
 Lie superalgebra $B(0,n)=\mathfrak{osp}(1,2n)$).

The finite WGRSs  were classified in~\cite{VGRS}, Sect. 7.
They consist of GRSs and two additional series $BC(m,n), C(m,n)$,
which can be described as follows. Let
$V$ be a complex vector space endowed with a symmetric bilinear form and an orthogonal  basis
$\{\vareps_i\}_{i=1}^m\cup\{\delta_j\}_{j=1}^n$
such that $||\vareps_i||^2=-||\delta_j||^2=1$ for
$i=1,\ldots,m, j=1,\ldots,n$.
One has
$$\begin{array}{l}
C(m,n)=\{\pm\vareps_i\pm\vareps_j;\pm 2\vareps_i \}_{1\leq i<j\leq m}
\cup\{\pm \delta_i\pm\delta_j; \pm 2\delta_i\}_{1\leq i<j\leq n}\cup\{\pm\vareps_i\pm\delta_j \}_{1\leq i\leq m}^{1\leq j\leq n},\\
BC(m,n)=C(m,n)\cup
\{\pm\vareps_i;\pm\delta_j \}_{1\leq i\leq m}^{1\leq j\leq n}.\end{array}$$

In particular, $C(1,1)$ is the root system of $\psl(2,2)$.

\subsubsection{Real roots of symmetrizable Kac-Moody algebras}
\label{exasymm}
Let $C$ be a symmetric $n\times n$ matrix with non-zero diagonal entries satisfying the condition $2c_{ij}/c_{ii}\in \mathbb{Z}$
for each $i,j$. Let $\Pi:=\{\alpha_1,\ldots,\alpha_n\}$
be a basis of a complex vector space $V$ and $(-,-)$ is
a symmetric bilinear form on $V$ given by $(\alpha_i,\alpha_j)=c_{ij}$.
Let $W$ be the subgroup of $GL(V)$ generated by $r_{\alpha_i}$ for
$i=1,\ldots,n$. Then $R(C):=W\Pi$ is a reduced GRRS without isotropic roots. If $C$ is such that $2c_{ij}/c_{ii}<0$ for each $i\not=j$,
then $C$ is a symmetric Cartan matrix and $R(C)$
is the set of real roots
of a symmetrizable Kac-Moody algebra
$\fg(C)$. Using the classification
of Cartan matrices in~\cite{K2} Thm. 4.3,
one readily sees that for a  symmetric Cartan matrix $C$,
$R(C)$ is affine and irreducible if and only if $\fg(C)$ is an affine Kac-Moody algebra.

Recall that
a basic classical Lie superalgebra $\fg\not=\mathfrak{psl}(n,n)$
is a symmetrizable Kac-Moody superalgebra and that
a finite-dimensional Kac-Moody superalgebra $\fg\not=\mathfrak{gl}(n,n)$ is a  basic classical Lie superalgebra. The root system of a finite-dimensional Kac-Moody superalgebra is a GRRS.
The set of real roots of a symmetrizable affine Kac-Moody superalgebra $\fg$
is an affine GRRS (with $\dim Ker (-,-)$ equals  $1$
if $\fg\not=\fgl(n,n)^{(1)}$ and equals  $2$ for
$\fgl(n,n)^{(1)}$); these algebras were classified by van de Leur in~\cite{vdL}.

Let $\fh$ be a Cartan subalgebra of a symmetrizable Kac-Moody algebra $\fg(C)$. Then $V$ is a subspace of $\fh^*$ spanned by $\Pi$. Take $\lambda\in\fh^*$ and define
$$\Delta_{re}(\lambda):=\{\alpha\in \Delta_{re}|\
\frac{2(\lambda,\alpha)}{(\alpha,\alpha)}\in\mathbb{Z}\}.$$
Then $\Delta_{re}(\lambda)$ is a subsystem of $\Delta_{re}$.

The above construction gives a reduced GRRS. Examples of non-reduced GRRSs without isotropic roots
can be obtained by the following procedure.
Fix $J\subset \{1,\ldots,n\}$ such that
$c_{ji}/c_{jj}\in\mathbb{Z}$ for each $j\in J, i\in \{1,\ldots,n\}$
and introduce
$$R(C)_J:=(\cup_{j\in J} W 2\alpha_j)\cup R.$$
It is easy to check that
 $R(C)_J$ is a GRRS (which is not reduced for $J\not=\emptyset$).
If $2c_{ij}/c_{ii}<0$ for each $i\not=j$, then $R(C)_J$
is  the set of real roots of a symmetrizable Kac-Moody superalgebra $\fg(C, J)$; as before $R(C)_J$ is affine and irreducible
if and only if $\fg(C, J)$ is affine.  By~\cite{H},
an indecomposable symmetrizable Kac-Moody superalgebra with an isotropic real root is finite-dimensional or affine. In~\Cor{corimR} we show that
an  irreducible GRRS which contains an isotropic root is either finite or affine.

\subsection{Quotients}\label{bijquo}
Let $R\subset V$ be a GRRS and
$V'$ be a subspace of $\Ker (-,-)$.
One readily sees that the image of $R$
in $V/V'$ satisfies the axioms (GR0), (GR2) and (WGR3).
We call this image a {\em quotient} of $R$ and
a {\em bijective quotient} if the restriction
of the canonical map $V\to V/V'$ to $R$ is injective.
The minimal quotient of $R$, denoted by $cl(R)$, is the image of $R$ in $V/Ker (-,-)$; by~\Cor{corimR} (i), $cl(R)$
is a WGRS.

\subsubsection{}\label{lem3}
Let $R\subset V$ be a GRRS, $V'$  be a subspace of $\Ker (-,-)$,
and $\iota: V\to V/V'$ be the canonical map.
Assume that the quotient $\iota(R)$ is a GRRS.
We claim that for any subsystem $R'\subset \iota(R)$,
the preimage of $R'$ in $R$, i.e. $\iota^{-1}(R')\cap R$,
is again a GRRS (and a subsystem of
$R$). The claim
 follows from the formula
$\iota(r_{\alpha}\beta)=r_{\iota(\alpha)} \iota(\beta)$
for each $\alpha,\beta\in R$
(note that $r_{\iota(\alpha)}$ is well-defined, since $\iota(R)$ is a GRRS).

\subsection{Direct sums}
Let $R_1\subset V_1$, $R_2\subset V_2$ be GRRSs.
Then $(R_1\cup R_2)\subset (V_1\oplus V_2)$ is again a GRRS.

Let $R=\cup_{i=1}^k R_i\subset V$, where $(R_i,R_j)=0$ for $i\not=j$,
and let $V_i$ be the span of $R_i$. Clearly, $R_i$ is a GRRS in $V_i$.
Since the natural map $\oplus_{i=1}^k V_i \to V$
preserves the form $(-,-)$, $R$ is a bijective quotient
of $\cup_{i=1}^k R_i\subset \oplus_{i=1}^k V_i$.
We conclude that any GRRS is a bijective quotient of
$\cup_{i=1}^k R_i\subset \oplus_{i=1}^k V_i$, where
$R_i\subset V_i$ are irreducible GRRSs. In particular,
if the form $(-,-)$ on $V$ is non-degenerate, then
$V=\oplus_{i=1}^k V_i$.

\subsection{Affinizations}\label{defaff}
Let $V$ be as above and $X\subset V$ be any subset.
Take $V^{(1)}=V\oplus\mathbb{C}\delta$ with the bilinear form $(-,-)'$
such that $\delta\in\Ker (-,-)'$ and the restriction
of $(-,-)'$ to $V$ coincides with the original form $(-,-)$ on $V$.
Set $X^{(1)}:=X+\mathbb{Z}\delta=\{\alpha+s\delta|\ \alpha\in X, s\in \mathbb{Z}\}$.

One readily sees that $R^{(1)}\subset V^{(1)}$ is a GRRS
if $R$ is a GRRS and $R$ is a quotient of $R^{(1)}$
in $V^{(1)}/\mathbb{C}\delta=V$.

We call  $R^{(1)}\subset V^{(1)}$ the {\em affinization} of  $R\subset V$
and use the notation $R^{(n)}\subset V^{(n)}$, where
 $R^{(n+1)}:=(R^{(n)})^{(1)}$,
$V^{(n+1)}:=(V^{(n)})^{(1)}$.

If $R$ is a finite GRRS, then $R^{(n)}$  is an  affine GRRS for any $n\geq1$.

Note that the affinizations of non-isomorphic GRRS can be isomorphic,
see~\Prop{propAnnx} (iii).

\subsection{Generators of a GRRS}
Let $R\subset V$ be a GRRS. Recall that a non-empty subset $X\subset R$ generates a subsystem
$R'$ if $R'$ is a unique minimal (by inclusion) subsystem of $R$ containing $X$.
If $R$ has no isotropic roots, then  any non-empty $X\subset R$ generates a unique subsystem, namely, $W(X)X$.
The following lemma gives a sufficient condition when $X$ generates a subsystem.

\subsubsection{}
\begin{lem}{lemGRS1}
Let $R\subset V$ be a GRRS.

(i) If $R'\subset R$ satisfies (GR2), (GR3), then
$R'':=R'\setminus (R')^{\perp}$ is either empty or is a GRRS.

(ii) If a non-empty $X\subset R$ is such that $X\cap X^{\perp}=\emptyset$, then
$X$ generates a subsystem $R'$ of $R$.
\end{lem}
\begin{proof}
(i) Let $R''$ be non-empty and let $V''$ be the span of $R''$.
 Let us verify that $R''\subset V''$ is a GRRS. Clearly, (GR1) holds.
If $x\in R''$, then $(x,y)\not=0$ for some $y\in R'$, so
$y\in R''$; thus $x$ is not in the kernel of the restriction
of $(-,-)$ to $V''$, so (GR0) holds. It remains to verify that for each
$\alpha,\beta\in R''$ one has $r_{\alpha}\beta\in R''$. Indeed,
since (GR2), (GR3) hold for $R'$, $r_{\alpha}\beta\in R'$.
If $(\alpha,\beta)=0$, then $r_{\alpha}\beta=\beta\in R''$;
otherwise $(r_{\alpha}\beta,\alpha)\not=0$
(for $(\alpha,\alpha)\not=0$ one has $(r_{\alpha}\beta,\alpha)=-(\beta,\alpha)$
and for $(\alpha,\alpha)\not=0$ one has $(r_{\alpha}\beta,\alpha)=(\beta,\alpha)$).
Hence $r_{\alpha}\beta\in R''$ as required.

(ii) By~\S~\ref{alphabeta}, for any $\alpha\in R$ the map $r_{\alpha}:R\to R$
satisfying (GR2), (GR3) respectively is uniquely defined.
Take
$$X_0:=X,\ \  X_{i+1}:=\{\pm r_{\alpha}\beta|\ \alpha,\beta\in X_i\},\ \
R':=\bigcup_{i=0}^{\infty} X_i.$$
Clearly, $R'$ satisfies (GR2), (GR3) and lies in any subsystem containing $X$.
Let us show $R'$ is a GRS. By (i), it is enough to
verify that $R'\cap (R')^{\perp}=\emptyset$.
Suppose that $v\in R'\cap (R')^{\perp}$; let $i$ be minimal such that
$v\in X_i$.  Since $X\cap X^{\perp}=\emptyset$ we have $i\not=0$, so
$v=r_{\alpha}\beta$ for some $\alpha,\beta\in X_{i-1}$ with $(\alpha,\beta)\not=0$. Since $\beta=r_{\alpha}v\not=v$,
one has $(\alpha,v)\not=0$, a contradiction.
One readily sees that $(v,\alpha)=\pm (\alpha,\beta)$, that is $(v,R')\not=0$, a contradiction.
\end{proof}

\subsubsection{}\label{DeltaPi}
Let $\fg$ be a basic classical Lie superalgebra, $\Delta\subset\fh^*$ be its roots system
and $\Pi\subset\Delta$ be a set of simple roots.
If $\fg\not=\mathfrak{psl}(n,n)$, $\Pi$ consists
of linearly independent vectors.
If $\fg\not=\mathfrak{osp}(1,2n)$ (i.e., $\Delta\not=BC_n=B(0,n)$),
then $\Delta$ is generated by $\Pi$.
We conclude that for $\fg\not=\mathfrak{psl}(n,n),
\mathfrak{osp}(1,2n)$, the root system $\Delta\subset V$ is generated
by a basis of $V$.

\section{The case when $(-,-)$ is non-degenerate}\label{sectnondeg}
In this section  $V\not=0$ is a finite-dimensional complex vector space
and $R\subset V$ satisfies (GR0), (GR2), (WGR3).
As before we say that $R\subset V$ is irreducible
if $R\not= R_1\coprod R_2$,
where $R_1,R_2$ are non-empty sets satisfying (GR2), (WGR3)
 and $(R_1,R_2)=0$.

We will prove the following theorem.

\subsection{}
\begin{thm}{thmdirsum}
Assume that the form $(-,-)$ is non-degenerate and $R\subset V$ satisfies (GR0), (GR2), (WGR3)
and

(GR1'): $R$ spans $V$.

Then

(i) If $R$ is irreducible and contains an isotropic root, then
$R$ is  finite (such $R$s are classified in~\cite{VGRS});

(ii) $R$ is a WGRS.
\end{thm}

\subsubsection{}
\begin{cor}{corimR}
(i) If $R$ is a GRRS, then the image of $R$ in $V/Ker (-,-)$
is a WGRS.

(ii) If $R$ is an irreducible GRRS which contains an isotropic root, then $R$ is either finite or affine.
\end{cor}

\subsubsection{}
\begin{rem}{}
By~\S~\ref{exasymm}, any symmetric $n\times n$ matrix $C$ with non-zero diagonal entries and
$2c_{ij}/c_{ii}\in\mathbb{Z}$ for each $i\not=j$, gives a GRRS.
Clearly, $(-,-)$ is non-degenerate if and only if $\det C\not=0$.
In this way we obtain a lot of examples of  infinite GRRSs with
non-degenerate $(-,-)$ (but they do not contain isotropic roots!).
\end{rem}

\subsection{Proof of~\Thm{thmdirsum}}\label{prthmdir}
We will use the following lemmas.

\subsubsection{}
\begin{lem}{lem2}
For any $\beta\in R$
there exists $\alpha\in R$ such that
$r_{\alpha}\beta$ is non-isotropic and
$(\beta,r_{\alpha}\beta)\not=0$.
\end{lem}
\begin{proof}
If $\beta$ is non-isotropic we take $\alpha:=\beta$.
Let $\beta$ be isotropic. Notice that $(\beta,r_{\alpha}\beta)=0$ implies $r_{\alpha}\beta=\pm \beta$, so it is enough
to show that $r_{\alpha}\beta$ is non-isotropic for some $\alpha\in R$. By (GR0) $\beta\not\in\Ker(-,-)$, so there exists $\gamma\in R$ such that $(\gamma,\beta)\not=0$, which implies
$(\beta,r_{\gamma}\beta)\not=0$. As a consequence,
one of the roots
$r_{\gamma}\beta$ or $r_{r_{\gamma}\beta}\beta$ is non-isotropic.
\end{proof}

\subsubsection{}
\begin{lem}{isoperp}
Let $R$ be  irreducible and contains an isotropic root.
For each $\alpha\in R$ there exists an isotropic root $\beta\in R$ with
$(\alpha,\beta)\not=0$.
\end{lem}
\begin{proof}
Let $R_{iso}\subset R$ be the set of isotropic roots.
Let   $R_2\subset R$ be the set of non-isotropic roots
in $R\cap R_{iso}^{\perp}$ and $R_1:=R\setminus R_2$.
One readily sees that $R_2$ is a subsystem.

Let us verify that $R_1$ is also a subsystem.
Indeed, let $\alpha,\beta\in R_1$
be such that
$(\alpha,\beta)\not=0$. One has
$(r_{\beta}\alpha,\beta)\not=0$, so
$r_{\beta}\alpha\in R_1$ if $\beta$ is isotropic.
If $\beta$ is non-isotropic, then, taking $\gamma\in R_{iso}$ such that
$(\alpha,\gamma)\not=0$, we get
$(r_{\beta}\alpha,r_{\beta}\gamma)=(\alpha,\gamma)\not=0$
and  $r_{\beta}\gamma\in R_{iso}$, that is
$r_{\beta}\alpha\in R_1$ as required.
Thus $\alpha,\beta\in R_1$ with
$(\alpha,\beta)\not=0$ forces $r_{\beta}\alpha\in R_1$.
Hence $R_1$ is a subsystem.

Suppose that there exist
$\alpha\in R_1$, $\beta\in R_2$ with $(\alpha,\beta)\not=0$.
By the construction of $R_2$, both $\alpha,\beta$ are non-isotropic.
 Since $(\alpha,\beta)\not=0$  one has $r_{\alpha}\beta=\beta+x\alpha$
for some $x\not=0$. Taking $\gamma\in R_{iso}$ such that
$(\alpha,\gamma)\not=0$, we get
$(r_{\alpha}\beta,\gamma)\not=0$ (since $(\beta,\gamma)=0$),
so $r_{\alpha}\beta\in R_1$. Since
$\alpha$ is non-isotropic and $R_1$ is a subsystem, one has $\beta=r_{\alpha}(r_{\alpha}\beta)\in R_1$, a contradiction.

We conclude that $R=R_1\coprod R_2$ with $(R_1,R_2)=0$.
Since $R$ is irreducible, $R_2$ is empty. This implies the assertion of the lemma
for non-isotropic root $\alpha$.

In the remaining case $\alpha\in R$ is isotropic.
Since $\alpha\not\in Ker (-,-)$, there exists $\gamma\in R$ such that
$(\gamma,\alpha)\not=0$. If $\gamma$ is isotropic, take $\beta:=\gamma$;
if $\gamma$ is non-isotropic, take $\beta:=r_{\gamma}\alpha$.
The assertion follows.
\end{proof}

\subsubsection{}
\begin{cor}{cor1234}
Let $R$ be irreducible and contains an isotropic root. If $\alpha\in R$
is non-isotorpic, then for each $\gamma\in R$ one has
$\ \ k_{\alpha,\gamma}\in \{0,\pm 1,\pm 2,\pm
3,\pm 4\}\ \ $
and $\ \ k_{\alpha,\gamma}\in \{0,\pm 1,\pm 2\}\ \ $
if $\gamma$ is isotropic.
\end{cor}
\begin{proof}
Let $(\alpha,\gamma)\not=0$.
If $\gamma$ is isotropic and $\alpha+\gamma\in R$, then
$$
||\alpha+\gamma||^2=(\alpha,\alpha)(1+k_{\alpha,\gamma}),\ \
\frac{2(\alpha+\gamma,\gamma)}{||\alpha+\gamma||^2}=
\frac{k_{\alpha,\gamma}}{1+k_{\alpha,\gamma}},$$
so (GR2) gives $k_{\alpha,\gamma}\in\{-1,-2\}$.
If $\gamma$ is isotropic and $\alpha+\gamma\in R$, then
$k_{\alpha,\gamma}\in\{1,2\}$.

Let $\gamma$ be non-isotropic. By~\Lem{isoperp},
there exists an isotropic $\beta\in R$ such that $(\beta,\gamma)\not=0$.
Since $\beta$ and $r_{\gamma}\beta$ are isotropic,
one has $k_{\alpha,\beta},k_{\gamma,\beta},k_{\alpha,r_{\gamma}\beta}
\in \{0,\pm 1,\pm 2\}$
and $k_{\gamma,\beta}\not=0$. Combining~(\ref{kkk})
and $k_{\alpha,\gamma}\in\mathbb{Z}$, we obtain the required formula.
\end{proof}

\subsubsection{Proof of finiteness}
Let $R\subset V$ satisfy the assumptions of~\Thm{thmdirsum}.
Let us show that $R$ is finite.

By (GR1') $R$ contains a basis $B$ of $V$.
Since $(-,-)$ is non-degenerate, each $v\in V$ is determined
by the values $(v,b)$, $b\in B$. Thus in order to show that
$R$ is finite, it is enough to verify that
the set $\{(\alpha,\beta)|\ \alpha,\beta\in R\}$
is finite.
If $\alpha,\beta\in R$ are isotropic and $(\alpha,\beta)\not=0$,
then $r_{\alpha}\beta$ is non-isotropic and
$(r_{\alpha}\beta,\alpha)=(\beta,\alpha)$. Thus
$$\{(\alpha,\beta)|\ \alpha,\beta\in R\}=\{0\}\cup S, \text{ where } S:=\{(\alpha,\beta)|\ \alpha,\beta\in R, \ (\alpha,\alpha)\not=0\}.$$
Using~\Cor{cor1234} we conclude that the finiteness of $S$ is equivalent
to the finiteness of $N:=\{(\alpha,\alpha)|\ \alpha\in R\}$.
Let $X\subset R$ be a maximal linearly independent set of non-isotropic roots and let $\alpha$ be a non-isotropic root.
Then $\alpha$ lies in the span of $X$, so $(\alpha,\alpha)\not=0$
implies $(\alpha,\beta)\not=0$ for some $\beta\in X$.
One has $(\alpha,\alpha)/(\beta,\beta)=k_{\beta,\alpha}/
k_{\alpha,\beta}$. From~\Cor{cor1234} we get
$$N\subset \{0, a/b(\beta,\beta)|\ \beta\in X,\ a, b \in \{\pm 1 ,\pm 2,\pm 3,\pm 4\}\},$$
so $N$ is finite as required.\qed

\subsubsection{Proof of (GR1)}
It remains to verify that $R$ satisfies (GR1).
Since the form $(-,-)$ is non-degenerate, $(R,V)$ is a direct sum of its irreducible components: $V=\oplus_{i=1}^k V_i$, where $(V_i,V_j)=0$ for $i\not=j$, and $R=\coprod_{i=1}^k R_i$, where
$R_i$ spans $V_i$, $R_i$ is irreducible and satisfies (GR0), (GR2), (WGR3) for each $i=1,\ldots,k$
Thus without loss of generality we can (and will) assume that
$R$ is irreducible. Let us show that

\begin{equation}\label{Nn}
 (-,-) \text{ can be normalized in such a way that }
(\alpha,\beta)\in \mathbb{Q} \text{ for all }\alpha,\beta\in R.
\end{equation}

implies (GR1).
Indeed, let
$B=\{\beta_1,\ldots,\beta_n\}\subset R$ be a basis of $V$
and let $\alpha_1,\ldots,\alpha_k\in R$ be  linearly dependent.
For each $i$ write $\alpha_i=\sum_{j=1}^n y_{ij} \beta_j$.
Since $(-,-)$ is non-degenerate and
$(\alpha,\beta)\in \mathbb{Q}$ for each $\alpha,\beta\in R$, we
have $y_{ij}\in \mathbb{Q}$ for each $i,j$. Since
$\alpha_1,\ldots,\alpha_k$ are linearly dependent,
$\det Y=0$. By above, the entries of  $Y$ are rational, so
there exist a rational vector
$X=(x_i)_{i=1}^k$ such that $YX=0$. Then
$\sum_{i=1}^k x_i\alpha_i=0$, so $\alpha_1,\ldots,\alpha_k\in R$
 are linearly dependent over $\mathbb{Q}$.
Thus the natural map
$\mathbb{Z}R\otimes_{\mathbb{Z}} \mathbb{C}\to V$
is injective. By (GR1') it is also surjective, so (GR1) holds.

Assume that $R$ does not  contain  isotropic roots. Let us show that
we can normalize $(-,-)$ in such a way that~(\ref{Nn}) holds.
For $\alpha,\beta\in R$
one has $(\alpha,\alpha)/(\beta,\beta)=k_{\beta,\alpha}/
k_{\alpha,\beta}\in\mathbb{Q}$ if $(\alpha,\beta)\not=0$. From the irreducibility of $R$, we obtain $(\alpha,\alpha)/(\beta,\beta)\in\mathbb{Q}$.
Thus we can normalize the form $(-,-)$
in such a way that $(\alpha,\alpha)\in  \mathbb{Q}$
for each $\alpha\in R$; in this case $(\alpha,\beta)\in\mathbb{Q}$
for any $\alpha,\beta\in R$, so~(\ref{Nn}) holds.

Now assume that $R$  contains an  isotropic roots.
By above, $R$ is finite; such systems
were classified in~\cite{VGRS}. From this classification it follows
that $R$ satisfies~(\ref{Nn}) except for $R=D(2,1,a)$ with $a\not\in\mathbb{Q}$; thus (GR1) holds
for such $R$. For $R=D(2,1,a)$ (and, in fact, for each
$R\not=\mathfrak{psl}(n,n)$)
there exists $\Pi\subset R$ such that $R\subset\mathbb{Z}\Pi$
and the elements of $\Pi$ are linearly independent.
In this case, $\Pi$ is a basis of $V$ and $\mathbb{Z}R=\mathbb{Z}\Pi$.
Thus (GR1) holds.\qed

\section{The minimal quotient $cl(R)$}\label{sect3}
In this section $V$ is a complex $(l+k)$-dimensional vector space
endowed with a degenerate symmetric bilinear form $(-,-)$ with a $k$-dimensional kernel, and
$R\subset V$ is a GRRS. The map $cl$ is the canonical map
$V\to V/\Ker(-,-)$. By~\Cor{corimR} (i), $cl(R)$ is a WGRS in $V/\Ker(-,-)$.

\subsection{Gaps}\label{gaps}
Consider the case when $\dim Ker (-,-)=1$.
From (GR1) it follows that $\mathbb{Z}R\cap \Ker(-,-)=\mathbb{Z}\delta$
for some (may be zero) $\delta$.

For each $\alpha\in cl(R)$ one has
$(cl^{-1}(\alpha)\cap R)\subset \{\alpha'+\mathbb{Z}\delta\}$ for some
$\alpha'\in R$. If $\delta\not=0$, we call $g(\alpha)\in\mathbb{Z}_{\geq 0}$
the {\em gap} of $\alpha$ if
$$cl^{-1}(\alpha)\cap R=\{\alpha'+\mathbb{Z}g(\alpha)\delta\}$$
for some $\alpha'\in R$;
if $\delta=0$ we set $g(\alpha):=0$
(in this case $cl^{-1}(\alpha)\cap R$ contains
only one element).

Observe that the set of gaps is an invariant of the root system.
The gaps have the following properties:

(i) $g(\alpha)$ is defined for all non-isotropic $\alpha\in cl(R)$;

(ii) $g(\alpha)$ are $W(cl(R))$-invariant (if $g(\alpha)$ is defined, then
$g(w\alpha)$ is defined and
$g(w\alpha)=g(\alpha)$ for each $w\in W(cl(R))$);

(iii) if $\alpha,\beta\in cl(R)$ are non-isotropic, then
$k_{\alpha,\beta}g(\alpha)\in\mathbb{Z} g(\beta)$;

 (iv) if $cl(R)$ is a GRRS, then  $g(\alpha)$ are defined for all $\alpha\in cl(R)$ and $g(\alpha)$
 are $GW(cl(R))$-invariant (see~\S~\ref{Weyl} for notation).

 The properties  (i)--(iii) are standard (we give a short proof in~\S~\ref{lemgap});
   (iv) will be established  in~\Prop{propFalpha}.

\subsubsection{}\label{lemgap}
Let us show that $g(\alpha)$ satisfies (i)--(iii).
Fix a non-isotropic $\alpha'\in R$ and set
$$M:=\{k\in\mathbb{Z}| \ \alpha'+k\delta\in R\}.$$
For each $x,y,z\in Ker(-,-)$ and $m\in\mathbb{Z}$ one has
\begin{equation}\label{rxyz}
(r_{\alpha'+x}r_{\alpha'+y})^m(\alpha'+z)=\alpha'+2m(x-y)+z.
\end{equation}
Thus  for each $p,q,r\in M$ one has $2\mathbb{Z}(p-q)+r\subset M$.
Note that $0\in M$ (since $\alpha'\in R$).
Taking $q=0$ and $r=0$ or $r=p$, we get $\mathbb{Z}p\subset M$.
Hence $M=\mathbb{Z}k$ for some $k\in\mathbb{Z}$.
This gives (i). Combining
$r_{\alpha'}(\beta+p\delta)=r_{\alpha'}\beta+p\delta$ (for any $\beta\in R$)
and the fact that $W(R)$ is generated by the reflections $r_{\beta}$ with non-isotropic
$\beta\in R$, we obtain (ii).

For (iii) take $\alpha,\beta\in cl(R)$. Notice
 that $g(\alpha),g(\beta)$ are defined by (i);
 by (ii) $g(r_{\alpha}\beta)=g(\beta)$. Take $\alpha',\beta'\in R$
such that $cl(\alpha')=\alpha$ and $cl(\beta')=\beta$.
Since $r_{\alpha'+k\delta}(\beta')=r_{\alpha'}\beta'+a_{\alpha,\beta}k\delta$
we have  $a_{\alpha,\beta}g(\alpha)\in\mathbb{Z} g(\beta)$ as required.

\subsection{Construction of $R'\subset cl(R)$}\label{Falphastr}
Set
$$L:=\mathbb{Z}R\cap \Ker(-,-).$$

From (GR1) it follows that $L=\sum_{i=1}^s \mathbb{Z}\delta_i$,
for some linearly independent
$\delta_1,\ldots, \delta_s\in Ker(-,-)$.

Since $R$ spans $V$, there exists $X:=\{v_1,\ldots,v_l\}\subset R$
whose images form a basis of $V/Ker (-,-)$. We fix $X$ and identify
$V/\Ker (-,-)$ with the vector space $V'\subset V$ spanned
by $v_1,\ldots,v_l$; then $V=V'\oplus \Ker (-,-)$ and $cl:\ V\to V'$
is the projection; in particular, $cl(R)$ is a WGRS in $V'$.
The restriction of $(-,-)$
to $V'$ is non-degenerate, so by~\Lem{lemGRS1} (ii),
the set $X$ generates
a subsystem $R'$ in $R$.

\subsection{Construction of $F(\alpha)$}
 For each $\alpha\in V'$ we introduce
$$F(\alpha):=\{v\in \Ker (-,-)|\ \alpha+v\in R\}.$$
Notice that $F(\alpha)$ is non-empty if and only if $\alpha\in cl(R)$.
For each $\alpha\in cl(R)$ one has
\begin{equation}\label{eqFalpha}
F(\alpha)\subset L+\delta_{\alpha}\ \text{ for some }\delta_{\alpha}
\ \text{ where }  \delta_{\alpha}=0  \text{ iff } \alpha\in R'.
\end{equation}

\subsection{}
\begin{lem}{lemaff}
If $cl(R)\subset\mathbb{Z}X$, then $R\subset cl(R)+L=cl(R)^{(k)}$
for  $k:=\dim Ker(-,-)$ and $\dim Ker(-,-)=rank L$.
\end{lem}
\begin{proof}
Clearly,  $cl(R)+L=cl(R)^{(s)}$, where $s=rank L$.
Fix $\alpha\in R$ and set $\mu:=\alpha-cl(\alpha)$. Then
$\mu \in Ker (-,-)$ and $\mu\in \mathbb{Z}R$, since
$cl(R)\subset\mathbb{Z}R$. Therefore $\mu\in L$. This gives $R\subset cl(R)+L$.
Since $R$ spans $V$, $s=\dim Ker(-,-)$ as required.
\end{proof}

\subsubsection{}
\begin{prop}{propFalpha}
For each $\alpha\in cl(R)$ one has

(i) $F(-\alpha)=-F(\alpha)$;

(ii) $F(w\alpha)=F(\alpha)$ for all $w\in W(R')$;

(iii) $F(\alpha)=-F(\alpha)$ if $\alpha$ is non-isotropic;

(iv) if $cl(R)$ is a GRRS, then for each $\alpha\in R'$
one has $F(\alpha)=-F(\alpha)$ and
$F(w\alpha)=F(\alpha)$ for each $w\in GW(R')$.
\end{prop}
\begin{proof}
The inclusion $R'\subset R$ implies (ii). The formula
$R=-R$ implies (i); (iii) follows from (i) and (ii).

For (iv) let $R'$ be a GRRS. Let us show
that for each $\alpha,\beta\in R'$ one has $F(r_{\beta}\alpha)=F(\alpha)$.
Clearly, this holds if $r_{\beta}\alpha=\alpha$;
by (ii) this holds if  $\beta$ is non-isotropic.
Since $r_{\beta}$ is an involution, it is enough to verify that
\begin{equation}\label{mumu}
F(r_{\beta}\alpha)\subset F(\alpha)\ \text{ for isotorpic } \beta\in R'.
\end{equation}

Assume that $r_{\beta}\alpha=\alpha+\beta$.
Then $\alpha+\beta\in R'\subset cl(R)$, so $\alpha-\beta\not\in cl(R)$
by~\S~\ref{alphabeta} (since $cl(R)$ is a GRRS).
Take $v\in F(\alpha)$. Then $\alpha+v\in R$,
so $r_{\beta}(\alpha+v)\in R$. Since $cl(\alpha+v-\beta)=\alpha-\beta\not\in cl(R)$,
one has $r_{\beta}(\alpha+v)=\alpha+v+\beta$, so $v\in F(r_{\beta}\alpha)$
as required. The case $r_{\beta}\alpha=\alpha-\beta$ is similar.

Thus we established the formula $F(r_{\beta}\alpha)=F(\alpha)$ for $\alpha\not=\pm\beta$.
By~\Lem{lem2} there exists $\gamma\in R'$ such that
$r_{\gamma}\beta$ is non-isotropic. By (iii)
$F(r_{\gamma}\beta)=-F(r_{\gamma}\beta)$. By above, $F(\beta)=F(r_{\gamma}\beta)$
(since $r_{\gamma}\beta\not=\pm\beta$). Hence
$F(\beta)=F(-\beta)$. This completes the proof of~(\ref{mumu}) and of (iv).
\end{proof}

\subsubsection{}
\begin{lem}{}
For any $\alpha,\beta\in cl(R)$ such that $r_{\alpha}$
is well-defined and  $r_{\alpha}\beta=\alpha+\beta$ one has
\begin{equation}\label{eqFalphabeta}
\begin{array}{l}
F(\alpha+\beta)=F(\alpha)+F(\beta).
\end{array}
\end{equation}
\end{lem}
\begin{proof}
Observe that $F(\alpha+\beta)=F(\alpha)+F(\beta)$ holds if
 for any $x\in F(\alpha), y\in F(\beta)$ and $z\in F(\alpha+\beta)$ one has
$$
r_{\alpha+x}(\beta+y)=\alpha+\beta+x+y,\ \ r_{\alpha+x}(\alpha+\beta+z)=\beta+z-x.$$

If $\alpha$ is non-isotropic, then $\alpha+x$ is also non-isotropic and
the above formulae follow from the definition of $r_{\alpha+x}$.
If $\alpha$ is isotorpic and $r_{\alpha}$ is well-defined, then,
by~\S~\ref{alphabeta},
$\beta-\alpha,\beta+2\alpha\not\in cl(R)$ (since
$\beta,\alpha+\beta\in cl(R)$), which implies
the above formulae.
Thus~(\ref{eqFalphabeta}) holds.
\end{proof}

\subsubsection{}
\begin{cor}{cor1}
Assume that

$\bullet\ cl(R)=R'$;

 $\bullet\ cl(R)$ is a GRRS and $cl(R)\not=A_1$;

  $\bullet\ GW(R')$ acts transitively on $R'$.

Then $R=R'+L$, i.e. $R=(R')^{(s)}$.
\end{cor}
\begin{proof}

We claim that $R'$ contains two roots $\alpha,\beta$ with $r_{\alpha}\beta=\alpha+\beta$.
Indeed, if $R$ contains an isotropic root $\alpha$, then it also contains
$\beta$ such that $(\alpha,\beta)\not=0$, so $r_{\alpha}\beta\in\{\beta\pm\alpha\}$,
so one of the pairs $(\alpha,\beta)$ or $(-\alpha,\beta)$ satisfies
the required condition. If all roots in $R'$ are non-isotropic, then any non-orthogonal
$\alpha',\beta'\in R'$ generate a finite root system, that is one of $A_2,C_2, BC_2$
and such root system contains $\alpha,\beta$ as required.

By~\Prop{propFalpha} (iv) $F(\gamma)$ is the same for all $\gamma\in R'$
and $F(\gamma)=-F(\gamma)$.
Using (\ref{eqFalphabeta}) for the pair $\alpha,\beta$ as above, we
conclude that $F(\alpha)$ is a subgroup of $Ker (-,-)$, so $F(\alpha)=\mathbb{Z}F(\alpha)$.
This implies $\mathbb{Z} R\cap Ker (-,-)=F(\alpha)$, that is
$F(\alpha)=L$ as required.
\end{proof}

\section{Case when $cl(R)$ is finite and is generated by a
basis of $cl(V)$}\label{sect4}
In this section we classify the irreducible GRRSs $R$ with a finite $cl(R)$ generated by a basis of $cl(V)$.

Combining~\S~\ref{GRS}, \ref{DeltaPi},  we conclude
that if $(-,-)$ is non-degenerate, then
$R$ is a root system of a basic classical Lie superalgebra $\fg\not=\mathfrak{psl}(n,n), \mathfrak{osp}(1,2n)$, i.e. $cl(R)$
is from the following list:
\begin{equation}\label{ABCDE}
\begin{array}{l}
\text{classical root systems }: A_n,B_n,C_n,D_n, E_6,E_7,E_8, F_4, G_2,\\
A(m,n)\ m\not=n, B(m,n), m,n\geq 1; C(n), D(m,n),\ m,n\geq 2; D(2,1,a), F(4),  G(3).\end{array}
\end{equation}

If the form $(-,-)$ is degenerate, then $cl(R)$ is from the list~(\ref{ABCDE});
the classification is given by the following theorem.

\subsection{}
\begin{thm}{thmAGRS1}
Let $R\subset V$ be a GRRS and
$$k=\dim Ker (-,-)>0.$$

(i) If $cl(R)$ is one of the following GRRSs
$$\begin{array}{l}
A_n, n\geq 2, D_n, n\geq 4, E_6,E_7, E_8;\\
A(m,n)\ m\not=n; C(n), D(m,n), m\geq 2, n\geq 1;
D(2,1,a), F(4),  G(3),
\end{array}$$
then $R$ is the affinization of $cl(R)$:
$R=cl(R)^{(k)}$.

(ii) The isomorphism classes of GRRSs $R$ with $cl(R)=A_1$
are in one-to-one correspondence with the equivalence classes
of the subsets $S$ of  the affine space $\mathbb{F}_2^k$ containing an affine base of  $\mathbb{F}_2^k$, up to  affine automorphisms of $\mathbb{F}_2^k$.

(iii) If $cl(R)=G_2, F_4$, then $R$ is of the form $R(s)$ for $s=0,\ldots,k$
$$R(s):=\{\alpha+\sum_{i=1}^k \mathbb{Z}\delta_i|\ \alpha\in cl(R) \text{ is short}\}\cup \{\alpha+ \sum_{i=1}^s \mathbb{Z}\delta_i+\sum_{i=s+1}^k \mathbb{Z}r\delta_i|\
\alpha\in cl(R)\text{ is long}\},$$
where $\{\delta_i\}$ is a basis of $Ker(-,-)$ and $r=2$ for $F_4$ and $r=3$
for $G_2$. The GRRSs $R(s)$ are pairwise non-isomorphic.

(iv) The  GRRSs $R$ with $cl(R)=C_2$
are parametrized by the pairs $(S_1,S_2)$, where
$S_i$ are subsets of  the affine space $\mathbb{F}_2^k$ containing zero
such that

(1) $S_1$ contains an affine base of  $\mathbb{F}_2^k$,

(2) $S_1+S_2\subset S_1$.

Moreover, $R(S_1,S_2)\cong R(S_1',S_2')$ if and only if for $i=1,2$ one has
$S_i'=\phi(S_i)+a_i$, where $\phi\in GL(\mathbb{F}_2^k)$
and $a_1,a_2\in \mathbb{F}_2^k$ (so $v\mapsto \phi(v)+a_i$
is an affine automorphism of $\mathbb{F}_2^k$).

(v) The isomorphism classes of GRRSs $R$ with $cl(R)=B_n,C_n, n\geq 3$, $B(m,n), m,n\geq 1$ are in one-to-one correspondence with the equivalence classes of non-empty
subsets $S$ of the affine space  $\mathbb{F}_2^k$
up to  affine automorphisms of $\mathbb{F}_2^k$.

\end{thm}

\subsubsection{Remarks}
In (ii), (iv) we mean that $R(S)\cong R(S')$ for $S,S'\subset \mathbb{F}_2^k$  if and only  $S'=\psi(S)$ for some affine
automorphism $\psi$.

Notice that all above GRRSs are infinite (so affine).

Observe that (i) corresponds to the case when
$WG(cl(R))$ acts transitively on $cl(R)$ and $cl(R)\not=A_1$.
For $cl(R)=B_n,C_n$ $F_4, G_2$ and $B(m,n), m,n\geq 1$,
$cl(R)$ has two $GW(cl(R))$-orbits (see~\S~\ref{Weyl} for notation). We denote these orbits by
$O_1, O_2$, where $O_1$ (resp., $O_2$) is the set of short (resp., long)
roots for $C_n, F_4,G_2$ and
$O_1=D_m,D(m,n)$ for $B_m,B(m,n)$
respectively (and $O_2$ is the set of short roots in
both cases).

\subsection{Description of $R(S)$}
In order to describe the above  correspondences in (ii)--(iv)
between GRRSs and subsets
in $\mathbb{F}_2^k$ we fix a  free abelian group
 $L\subset Ker (-,-)$  of rank $k$ and
 denote by $\iota_2$  the canonical map
$\iota_2: L\to L/2L\cong\mathbb{F}_2^k$ and by $\iota^{-1}_2$
the preimage of $S\subset \mathbb{F}_2^k$ in $L$.

\subsubsection{Case when $S\subset \mathbb{F}_2^k$ contains zero}

For $cl(R)=A_1=\{\pm \alpha\}$ (case (ii)) one has
$$R(S):=\{\pm \alpha+\iota_2^{-1}(S)\}.$$
For $cl(R)=C_n$ one has
$$\begin{array}{ll}
R(S_1,S_2):=\{\alpha+\iota_2^{-1}(S_1)|\ \alpha\in O_1\}\cup
\{\alpha+ \iota_2^{-1}(S_2)|\ \alpha\in O_2\}\ & \text{ for  }n=2,\\
R(S):=\{\alpha+L|\ \alpha\in O_1\}\cup
\{\alpha+ \iota_2^{-1}(S)|\ \alpha\in O_2\} & \text{ for  }n>2.
\end{array}$$
In these cases $L:=\mathbb{Z}R\cap Ker(-,-)$.

For $cl(R)=B_n, n\geq 3,\ B(m,n), m,n\geq 1$ we take
$$R(S):=\{\alpha+2L|\ \alpha\in O_1\}\cup
\{\alpha+ \iota_2^{-1}(S)|\ \alpha\in O_2\}.$$

\subsubsection{}\label{choices}
Now assume that $S\subset \mathbb{F}_2^k$ is an arbitrary non-empty set.
Take any $s\in S$ and consider a set $S(s):=S-s$. The set $S(s)$ contains zero (and contains an affine basis for $\mathbb{F}_2^k$ if $S$
contained such a basis), so $R(S-s)$ is defined above. We set $R(S,s):=R(S-s)$. The sets
 $S(s)$ (for different choices of $s$) are conjugated by affine automorphism, so, as we will show in~\S~\ref{isoRS},
 the GRRSs corresponding to different choices of $s$ are isomorphic:
 $R(S,s')\cong R(S,s'')$ for any $s',s''\in S$ (in other words,
 $R(S):=R(S,s)$ is defined up to an isomorphism).

\subsection{Proof of~\Thm{thmAGRS1}}
The rest of this section is devoted to the proof of~\Thm{thmAGRS1}.
We always assume that $0\in S$.
Considering $B_n$ (reps., $B(m,n)$) we always assume that $n\geq 3$ (resp., $m,n\geq 1$).

\subsubsection{}\label{pfAGRS11}
Recall that $cl(R)$ is generated (as a GRRS) by a basis $\Pi$ of $cl(V)$. We take
 $X:=\Pi$ in the construction of $R'$ (see~\S~\ref{Falphastr}).
We obtain $R'=cl(R)$, so $cl(R)\subset R$.
Using~\Lem{lemaff} we obtain
$$cl(R)\subset R\subset cl(R)^{(k)}.$$

\subsubsection{}
It is easy to verify that
if $cl(R)$ is as in (i), then $WG(cl(R))$ acts transitively on $cl(R)$, so (i) follows from~\Cor{cor1}.

\subsubsection{Case $cl(R)=A_1$}
Let $cl(R)=A_1=\{\pm\alpha\}$; set $L:=\mathbb{Z}R\cap Ker (-,-)$.
Then $R\subset cl(R)^{(k)}$ and so, by (GR1), $L$ is a free group
of rank $k$. If $k=1$, then $R=A_1^{(1)}$ by~\S~\ref{gaps}.
Consider the case $k>1$. Recall that $R=\{\pm (\alpha+H)\}$, where $H\subset L$ contains $0$, so (GR1) is
equivalent to the condition that $H$ contains a basis of $L$.
For each $x,y\in Ker (-,-)$ one has $r_{\alpha+x}(\alpha+y)=-\alpha+y-2x$,
 so (GR2) is equivalent to $2x-y\in H$ for each $x,y\in H$, that is $H+2L\subset H$.
Hence $H$ is a set of equivalence classes of $L/2L=\mathbb{F}_2^k$ which contains $0$ and a basis of $\mathbb{F}_2^k$.

View $\mathbb{F}_2^k$ as an affine space. Recall that an affine basis of a $k$-dimensional
affine space $\mathbb{F}^k$ is a collection of points $x_1,\ldots,x_{k+1}$
such that any point $y\in \mathbb{F}^k$ is of the form $\sum_{i=1}^{k+1} \lambda_i x_i$
for some $\lambda_i\in\mathbb{F}$ with $\sum_{i=1}^{k+1}\lambda_i=1$.
We conclude that $R=\{\pm (\alpha+H)\}$ is a GRRS if and only if
the set $S:=\iota_2(H)\subset\mathbb{F}_2^k$ has the following properties: $0\in S$ and $S$
contains an affine basis  of $\mathbb{F}_2^k$.

\subsubsection{Construction of $H_1,H_2$}
Assume that $WG(cl(R))$ does not act transitively on $cl(R)$. Then $cl(R)$ has two orbits $O_1$ and $O_2$, see above.
By~\Prop{propFalpha} one has
$$R=\{\alpha+H_1|\ \alpha\in O_1\}\cup \{\alpha+H_2|\ \alpha\in O_2\},$$
where $H_1,H_2\subset Ker(-,-)$  and $0\in H_1, H_2$ (since $cl(R)\subset R)$.

Except for the case $cl(R)=C_2$ the orbit $O_1$ is an irreducible  GRRS with the transitive action of
$WG(O_1))$ (one has $O_1=D_n$ for $B_n,C_n, F_4$, $O_1=D(m,n)$ for $B(m,n)$ and $O_1=A_2$ for $G_2$).
Combining~\Lem{lem3} and (i), we obtain
that $H_1$ is a free abelian subgroup of $Ker(-,-)$ if
$cl(R)\not=C_2$. We introduce $L$ as follows:
\begin{equation}\label{eqL}
L:=\left\{ \begin{array}{ll}
H_1 & \text{ for }cl(R)\not=C_2, B(m,n), B_n;\\
\frac{1}{2}H_1 & \text{ for }cl(R)=B(m,n), B_n;\\
\mathbb{Z}R\cap Ker(-,-) & \text{ for }cl(R)=C_2.
\end{array}\right.\end{equation}

\subsubsection{Cases $F_4, G_2$}
For these cases  $O_2\cong O_1$, so $O_2$ is also
an irreducible  GRRS with the transitive action of
$WG(O_2))$, and thus $H_2$ is a free abelian subgroup of $Ker(-,-)$.
One readily sees that (GR2) is equivalent to
$$H_2+rH_1, H_2+H_2\subset H_2,\ \ H_1+H_2\subset H_1,$$
where $r=2$ for $F_4$ and $r=3$ for $G_2$. This gives
$rL\subset H_2\subset L$, so $H_2/(rL)$ is an additive subgroup of $\mathbb{F}_r^k$. Thus $H_2/(rL)\cong \mathbb{F}_r^s$ for some $0\leq s\leq k$ and $s$ is an invariant of $R$.
This establishes (iii).

\subsubsection{Case $C_n$}
Take $n>2$.
One readily sees that (GR2) is equivalent to
$$H_2+2H_1, H_2+2H_2\subset H_2,\ \ H_1+H_2\subset H_1.$$
Since $H_1=L$, we get $H_2+2L\subset H_2\subset L$.
Taking $S:=\iota_2(H_2)$, we obtain $R\cong R(S)$.

Take $n=2$. In this case (GR2) is equivalent to
$$H_1+H_2, H_1+2H_1\subset H_1,\  H_2+2H_1, H_2+2H_2\subset H_2.$$
 Since $0\in H_1$, we obtain $H_2\subset H_1$, so $L=\mathbb{Z}R\cap Ker(-,-)$ is spanned by $H_1$.
Thus (GR2) is equivalent to $H_i+2L\subset H_i$ for $i=1,2$
and $H_1+H_2\subset H_1$. Taking $S_i:=\iota_2(H_i)$
 for $i=1,2$, we obtain $R\cong R(S_1,S_2)$ as required.

\subsubsection{Cases $B_n, B(m,n)$}
One readily sees that (GR2) is equivalent to
$$H_2+2H_2, H_2+H_1\subset H_2,\ \ H_1+2H_2\subset H_1.$$
Since $H_1=2L$, we get $H_2\subset L$ and $H_2+2L\subset H_2$.
Taking $S:=\iota_2(H_2)$, we obtain $R\cong R(S)$.

 \subsection{Isomorphisms $R(S)\cong R(S')$}\label{isoRS}
 It remains to verify that in (ii)-(v) one has
 $R(S)\cong R(S')$ if and only if
 $S=\psi(S')$ for some affine transformation $\psi$
 (for $C_2$ we have $S_i=\psi(S'_i)$ for $i=1,2$).

\subsubsection{}
Let $R(S)\subset V,\ R(S')\subset V'$ be two isomorphic GRRSs and let $\phi: V\iso V'$
with  $\phi(R(S))=R(S')$ be the isomorphism.
Define $L,L'$ and
$H_i, H_i' \ (i=1,2)$
for $R(S)$ and $R(S')$ as above (for $cl(R)=A_1$ we set $O_1:=O_2:=A_1$ and $H_1:=H_2:=H$). From~(\ref{eqL}) one has
 $\phi(L)=L'$ and thus $\phi(2L)=2L'$, so $\phi$
induces a linear isomorphism $\phi_2:  \mathbb{F}_2^k\iso\mathbb{F}_2^k$
such that $\iota_2'\circ \phi=\phi_2\circ \iota_2$
(where $\iota_2: L/2L\iso \mathbb{F}_2^k$ and
$\iota_2': L'/2L'\iso \mathbb{F}_2^k$ are the natural isomorphisms).

By the above construction,
$R(S)$ and $R(S')$ contain $cl(R(S))\cong cl(R(S'))$.
Take $\alpha\in O_2\in cl(R(S))$ and let $\alpha'$ be the corresponding element in
$cl(R(S'))$.
Then $\phi(\alpha)=\alpha'+v$ for some $v\in H'_2$. Since
$\phi$ is linear,
$\phi(\alpha+x)=\alpha'+v+\phi(x)$ for each $x\in L$.
This implies $H_2'=v+\phi(H_2)$, that is
$$S'=\iota'_2(H'_2)=\iota'_2(v)+\iota'_2(\phi(H_2))=\iota'_2(v)+
\phi_2(\iota_2(H_2))=\iota'_2(v)+\phi_2(S).$$
This shows that $S'$ is obtained from $S$ by an affine automorphism $\psi:=\iota'_2(v)+\phi_2$
of $\mathbb{F}_2^k$ as required.

 For the case $cl(R)=C_2$ the above argument gives
 $S'_i=a_i+\phi_2(S_i)$ for some $a_i\in S_i'$ ($i=1,2$).

 \subsubsection{}
 Let $R(S), R(S')\subset V$ be two GRRSs with $cl(R(S))=cl(R(S'))$ (and the same
 $L$), and let $S'=\psi_2(S)+\ol{a}$ if $cl(R)\not=C_2$
 (resp., $S_i=\psi_2(S)+\ol{a}_i$ for $i=1,2$ if $cl(R)\not=C_2$), where $a\in L$ and $\ol{a}\in \mathbb{F}_2^k=L/2L$ (resp., $a_i\in L$ and $\ol{a}\in \mathbb{F}_2^k$)
 and $\psi_2$ is a linear automorphism of $\mathbb{F}_2^k$.
 Fix a linear isomorphism $\psi:L\to L$ such that
 $\iota_2\circ \psi=\psi_2\circ \iota_2$.

 Recall that $V=Ker(-,-)\oplus \mathbb{C}\Pi$, where
 $Ker (-,-)=L\otimes_{\mathbb{Z}}\mathbb{C}$ and
$\Pi\subset cl(R(S))=cl(R(S'))$ is linearly independent in $V$.
 Extend $\psi$ to a linear automorphism of $V$ by
 putting $\psi(\alpha):=\alpha+a$ for each $\alpha\in \Pi\cap O_2$
 and $\psi(\alpha):=\alpha$ for each $\alpha\in \Pi\cap O_1$ if
 $cl(R)\not=A_1,C_2$
 (resp., $\psi(\alpha):=\alpha+a$ for $\alpha\in \Pi$ if $cl(R)=A_1$ and
$\psi(\alpha):=\alpha+a_i$ for each $\alpha\in \Pi\cap O_i$, where $i=1,2$
for $cl(R)=C_2$).
 One readily sees that $\psi$ preserves $(-,-)$ and $\psi(R(S))=R(S')$.
 Thus $R(S)\cong R(S')$ (resp., $R(S_1,S_2)\cong R(S'_1,S'_2)$) as required.

\section{Case when $cl(R)$ is the root system
of $\mathfrak{psl}(n+1,n+1)$ for $n>1$}
\label{Ann}
In this section we describe $R$ such that $cl(R)$ is the root system
of $\mathfrak{psl}(n+1,n+1)$ for $n>1$.

\subsection{Description of $A(n,n), A(n,n)_f, A(n,n)_x$}

A finite  GRRS $A(n,n)\subset V$ can be described as follows.
Let $V_1$ be a complex vector space endowed
with a symmetric bilinear form
 and an orthogonal basis $\vareps_1,\ldots,\vareps_{2n+2}$ such that $(\vareps_i,\vareps_i)=-(\vareps_{n+1+i},\vareps_{n+1+i})=1$ for $i=1,\ldots, n+1$. One has

$$A(n,n)=\{\vareps_i-\vareps_j\}_{i\not=j},\ \ \ V=\{\sum_{i=1}^{2n+2}
a_i\vareps_i|\ \sum_{i=1}^{2n+2} a_i=0\},$$
where the reflection $r_{\vareps_i-\vareps_j}$ is
the restriction of the linear map
$\tilde{r}_{\vareps_i-\vareps_j}\in End(V_1)$
which interchanges $\vareps_i\leftrightarrow \vareps_j$
and preserves all other elements of the basis.
One readily sees that $A(n,n)\subset V$ is a finite GRRS; it is
the root system of the Lie superalgebra $\mathfrak{pgl}(n+1,n+1)$
($V_1$ corresponds to $\fh^*$, where $\fh$ is a Cartan subalgebra of
$\mathfrak{gl}(n+1,n+1)$ and $V\subset V_1$ is dual to the Cartan subalegbra
of $\mathfrak{pgl}(n+1,n+1)$).
The kernel of the bilinear form on $V$ is spanned by
$$I:=\sum_{i=1}^{n+1}(\vareps_i-\vareps_{n+1+i}).$$

The root system of $\mathfrak{psl}(n+1,n+1)$ is the quotient of $A(n,n)$
by $\mathbb{C}I$ (it is a bijective quotient if $n>1$);
we denote it by $A(n,n)_f$: $A(n,n)_f:=cl(A(n,n))$. Recall that
$A(n,n)_f$ is a GRRS if and only if $n>1$ and
 $A(1,1)_f$ is a WGRS (denoted by $C(1,1)$ in~\cite{VGRS}, see~\S~\ref{GRS}).

Let $A(n,n)^{(1)}\subset V^{(1)}:=V\oplus \mathbb{C}\delta$ be the affinization of $A(n,n)$. We denote by $cl_x$ the canonical map
 $$cl_x: V\oplus\mathbb{C}\delta\to V_x:=(V\oplus\mathbb{C}\delta)/\mathbb{C}(I+ x\delta)$$
and by  $A(n,n)_x$ the corresponding quotient of $A(n,n)^{(1)}$:
$$A(n,n)_x:=cl_x(A(n,n)^{(1)}).$$
Note that $A(n,n)_0=(A(n,n)_f)^{(1)}$.
The kernel of $(-,-)$ on $V_x$ is one-dimensional and $$cl(A(n,n)_x)\cong A(n,n)_f.$$
Note that (GR1) holds for
$A(n,n)_x$ only if $x\in\mathbb{Q}$ (since $cl_x(\delta),
xcl_x(\delta)\in \mathbb{Z}A(n,n)_x$). We will see that
for $n>1$ this condition is sufficient: $A(n,n)_x$ is a GRRS if and only if
$x\in\mathbb{Q}$; for $n=1$ integral values of $x$ should be excluded, see~\ref{uksi}
below.

\subsection{Description of $A(1,1)_x, \ x\in\mathbb{Q}$}\label{uksi}
Let $x=p/q$ be the reduced form ($p,q\in\mathbb{Z}, q>1, GCD(p,q)=1$).
Set $\delta':=cl_x(\delta)/q,\  e:=cl_x(\vareps_1-\vareps_2)/2,\
d:=cl_x(\vareps_3-\vareps_4
)/2$; note that $\delta', e,d$ form
 an orthogonal basis of $V_x$ satisfying $(\delta',\delta')=0$ and
$(e,e)=-(d,d)=1/2$. One has
$$A(1,1)_x=\{\pm 2e+\mathbb{Z}q\delta',\ \pm 2d+\mathbb{Z}q\delta',\
\pm e\pm d+ (\mathbb{Z}q\pm p/2)\delta'\},$$
and $cl(A(1,1)_x)=A(1,1)_f=C(1,1)=\{\pm 2e,\ \pm 2d,\ \pm e\pm d\}$.

Note that $\mathbb{Z}A(1,1)_x\cap Ker (-,-)$ is $\mathbb{Z}\delta'$
(since $GCD(p,q)=1$), so the non-isotropic roots in $C(1,1)$
has the gap $q$ (and the gap of isotropic roots is not defined).
Observe that $A(1,1)_x$ is not a GRRS for $x\in\mathbb{Z}$, since
$\alpha:=e+d+p/2\delta', \beta:=e-d-p/2\delta'$ are isotropic non-orthogonal roots and
$\alpha\pm\beta\in R$ which contradicts to (GR3), see~\S~\ref{alphabeta}.

\subsection{}
In this section we prove the following proposition describing the
affine GRRSs $R$ with $cl(R)=A(n,n)_f, n>1$.

\begin{prop}{propAnnx}
(i) $A(1,1)_x$ is a GRRS if and only if $x\in\mathbb{Q},\ x\not\in\mathbb{Z}$;
$A(n,n)_x$ for $n>1$ is a GRRS if and only if $x\in\mathbb{Q}$.

(ii) Let $R$ be a GRRS with $\dim Ker(-,-)=1$ and $cl(R)=A(n,n)_f, n>1$.
If $R$ is finite,
then $R\cong A(n,n)$. If $R$ is infinite, then $R\cong A(n,n)_x$ for some
$x\in\mathbb{Q}$ and it is a bijective quotient of $A(n,n)^{(1)}$;
each $\alpha\in A(n,n)_f$ has the gap $q$.
Moreover $A(n,n)_x\cong A(n,n)_y$ if and only if either $x+y$ or $x-y$ is integral.

(iii) If $R$ is a GRRS with $\dim Ker(-,-)=k+1>1$ and $cl(R)=A(n,n)_f, n>1$, then
$R$ is isomorphic to
 $A(n,n)^{(k+1)}$ or to its bijective quotient
$A(n,n)_{1/q}^{(k)}$ for some $q\in\mathbb{Z}_{>0}$ and these GRRS are
pairwise non-isomorphic. Moreover,
$A(n,n)_{p/q}^{(k)}\cong A(n,n)_{1/q}^{(k)}$ if $GCD(p,q)=1$.
\end{prop}

\subsubsection{Remark}
Recall that $A(n,n)_{0}=A(n,n)_f^{(1)}$, so  for each $p\in\mathbb{Z}$ one has
$A(n,n)_f^{(k+1)}\cong A(n,n)_p^{(k)}$ for $k\geq 0$.

\subsection{Proof}
By above, $A(n,n)_x$ satisfies (GR1) only if
$x\in\mathbb{Q}$ and, in addition, $x\not\in\mathbb{Z}$
for $n=1$. One readily sees that the converse holds
(these conditions imply (GR1)).
Since $A(n,n)^{(1)}$ is a GRRS, its quotient
$A(n,n)_x$ satisfies (GR0), (GR2) and (WGR3).
Using~\S~\ref{alphabeta} it is easy to show that
(GR3) does not hold if and only if
$n=1$ and $x\in\mathbb{Z}$. This establishes (i).

It is easy to see that $A(n,n)_x$ is a bijective quotient of
$A(n,n)^{(1)}$ for $n>1$.

\subsubsection{}
Let $R$ be a GRRS with  $cl(R)=A(n,n)_f, n>1$.
Set $L:=Ker (-,-)\cap \mathbb{Z}R$; by (GR1) one has $L\cong \mathbb{Z}^{k+1}$,
where $k+1=\dim Ker(-,-)$.

Recall that $\tilde{\Pi}:=\{\vareps_i-\vareps_{i+1}\}_{i=1}^{2n+1}$
is a set of simple roots for $A(n,n)$ and
$\Pi:=\{\vareps_i-\vareps_{i+1}\}_{i=1}^{2n}$
is a set of simple roots for a GRRS $A(n,n-1)$. Applying the procedure described in~\S~\ref{Falphastr} to $X:=\Pi$, we get $R'=A(n,n-1)$.
Let $V'$ be the span of $R'$. One has $V=\mathbb{C} I\oplus V'$, so
$R'=A(n,n-1)$ can be naturally viewed as a subsystem of $A(n,n)_f$.
Note that $A(n,n)_f$ has three $GW(A(n,n-1))$-orbits: $A(n,n-1)$ itself, $O_1:=\{\vareps_i-\vareps_{2n}\}_{i=1}^{2n-1}$
and $-O_1$. By~\Prop{propFalpha} for $i\not=j<2n$ one has
$$F(\vareps_i-\vareps_j)=L',\ \ F(\pm(\vareps_i-\vareps_{2n}))=\pm S,$$
where $S,L'\subset L$ and, by~\Thm{thmAGRS1} (since $n>1$),
$L'$ is a free group. By~(\ref{eqFalphabeta}),
$$S=L'+S,\ \ S+(-S)=L',$$
so $S=a+L'$ for some $a\in L$. Note that $L=L'+\mathbb{Z}a$.

If $a\not\in \mathbb{Q}L'$, then $L=L'\oplus \mathbb{Z}a$.
Extending the embedding $A(n-1,n)\to A(n,n)$ by $a\mapsto I$
we obtain the isomorphism $R\cong A(n,n)^{(k)}$.
(If $k=0$, then $L'=0$, so $R\cong A(n,n)$).

If $a\in L'$, then $S=-S=L$ and $R=(A(n,n)_f)^{(k+1)}=(A(n,n)_0)^{(k)}$.

Consider the remaining case $a\in \mathbb{Q}L'\setminus L'$.
Take the minimal $q\in\mathbb{Z}_{>1}$ such that $qa\in L'$  and
the maximal $p\in\mathbb{Z}_{>0}$ such that $qa\in pL$. Then
$GCD(p,q)=1$ and
\begin{equation}\label{LL'}
L'=\mathbb{Z} e\oplus L'',\ S=(p/q+\mathbb{Z})e\oplus L'',\
L=\mathbb{Z} \frac{e}{q}\oplus L''
\ \text{ where }
L''\cong \mathbb{Z}^k\end{equation}
where $e:=\frac{q}{p}a$. Hence
$$R\cong\bigl(A(n,n)/(I-\frac{p}{q}\delta)\bigr)^{(k)}
=\bigl(A(n,n)_{p/q}\bigr)^{(k)}.$$

\subsubsection{}
Let us show that
$A(n,n)_x\cong A(n,n)_y$ if either $x+y$ or $x-y$ is integral.
Consider the linear endomorphisms $\psi, \phi\in End(V\oplus\mathbb{C}\delta)$
defined by
$$\psi(v):=v\ \text{ for }v\in V;
\ \ \psi(\delta):=-\delta,$$
and $\phi(\delta)=\delta$,
$$\phi(\vareps_i-\vareps_{i+1})=\vareps_i-\vareps_{i+1}\ \text{ for }
i=1,\ldots,2n;\  \phi(\vareps_{2n+1}-\vareps_{2n+2})=\vareps_{2n+1}-\vareps_{2n+2}+\delta.
$$
These endomorphisms preserve $(-,-)$ and $A(n,n)^{(1)}$.
Since $\psi(I+x\delta)=I-x\delta$ and $\phi(I+x\delta)=I+(x+1)\delta$,
$\psi$ (resp., $\phi$) induces an isomorphism $V_x\to V_{-x}$
(resp., $V_x\to V_{x+1}$)
which preserves
the bilinear forms and maps $A(n,n)_x$ to $A(n,n)_{-x}$
(resp., to $A(n,n)_{x+1}$).
Hence $A(n,n)_x\cong A(n,n)_{-x}\cong A(n,n)_{x+1}$ as required.

Let us show that
$A(n,n)_x\cong A(n,n)_y$ imples that either $x+y$ or $x-y$ is integral.
For each subset $J$ of  $A(n,n)_x$ we
set $sum(J):=\sum_{\alpha\in J}\alpha$ and we let $U$ be the set
of subsets $J$ of $A(n,n)_x$ containing exactly $n+1$ roots.
It is not hard to see that
\begin{equation}\label{Upq}
Ker (-,-)\cap \{sum(J)|\ J\in U\}=\left\{\begin{array}{ll}
(\pm p+\mathbb{Z}q)\delta' & \text{ for even } n,\\
(\pm p+\mathbb{Z}q)\delta' \cup \{\mathbb{Z}q\delta'\}& \text{ for odd } n,
\end{array}\right.
\end{equation}
where $\delta'$ is a generator of $\mathbb{Z}R\cap Ker(-,-)\cong \mathbb{Z}$ and
$x=p/q$ with $GCD(p,q)=1$. Thus $A(n,n)_x\cong A(n,n)_y$ implies
$\pm p+\mathbb{Z}q=\pm p'+\mathbb{Z}q'$, where $y=p'/q'$ with $GCD(p',q')=1$.
The claim follows; this completes the proof of (ii).

\subsubsection{}
Now take $R$ such that $cl(R)=A(n,n)_f$ with $n>1$ and
fix $\alpha\in R$. Set $L:=Ker(-,-)\cap \mathbb{Z}R$
and $L':=\{v\in Ker(-,-)| \alpha+v\in R\}$.
One readily sees from above that $L/L'=\mathbb{Z}$
if $R=A(n,n)^{(k)}$ and
$L/L'=\mathbb{Z}/q\mathbb{Z}$ if $R=A(n,n)_{p/q}^{(k)}$ (with $GCD(p,q)=1$).
Therefore $A(n,n)^{(k)}\not\cong A(n,n)_{p/q}^{(k')}$ and
$A(n,n)_{p/q}^{(k)}\cong A(n,n)_{p'/q'}^{(k')}$ with $GCD(p',q')=1$ forces $q=q', k=k'$.

It remains to check that
for $k\geq 1$ one has  $A(n,n)_{p/q}^{(k)}\cong A(n,n)_{1/q}^{(k)}$.
Clearly, it is enough to verify this for $k=1$. Note that
$A(n,n)_x^{(1)}$ is the quotient of $A(n,n)^{(2)}$ by $\mathbb{C}(I+x\delta)$:
$A(n,n)\subset V$ and
$$V^{(2)}=V\oplus (\mathbb{C}\delta\oplus\mathbb{C}\delta'),\ \
A(n,n)^{(2)}=A(n,n)+\mathbb{Z}\delta+\mathbb{Z}\delta',$$
where $(V^{(2)},\delta)=(V^{(2)},\delta')=0$.
Consider the linear endomorphism $\phi\in End(V^{(2)})$
defined by $\phi(\delta)=a\delta+q\delta'$, $\phi(\delta')=b\delta+p\delta'$
where $a,b\in\mathbb{Z}$ are such that $pa-qb=1$ and
$$\phi(\vareps_i-\vareps_{i+1})=\vareps_i-\vareps_{i+1}\ \text{ for }
i=1,\ldots,2n;\  \phi(\vareps_{2n+1}-\vareps_{2n+2})=\vareps_{2n+1}-
\vareps_{2n+2}-b\delta-p\delta'.
$$
Then $\phi(I)=I-b\delta-p\delta'$, so
$\phi(I+\frac{p}{q}\delta)=I+\frac{1}{q}\delta$ and
$\phi$  induces an isomorphism $A(n,n)_{p/q}^{(k)}\cong A(n,n)_{1/q}^{(k)}$.
This completes the proof of (iii).

\section{The cases $BC_n, BC(m,n), C(m,n)$}\label{sect6}
\subsection{Case $BC_n$}
Let $cl(R)=BC_n$ and $k=\dim Ker(-,-)$.

Let $\{\vareps_i\}_{i=1}^n$ be an orthonormal basis of $cl(V)$.
Recall that $cl(R)=BC_n$  have three $W(BC_n)$-orbits
$$O_1:=\{\pm\vareps_i\}_{i=1}^n,\ \ O_2:=\{\pm 2\vareps_i\}_{i=1}^n,\ \
O_3:=\{\pm\vareps_i\pm\vareps_j\}_{1\leq i<j\leq n}^n$$
for $n>1$ and two $W$-orbits, $O_1$ and $O_2$, for $n=1$.

We take $X$ to be a set of simple roots of $B_n=O_1\cup O_3$
($X=\{\vareps_1-\vareps_2,\ldots,\vareps_{n-1}-\vareps_n,\vareps_n\}$)
 in the construction of $R'$ (see~\S~\ref{Falphastr}).
 Then $R'=B_n$ and $W(BC_n)=W(B_n)$.
  We set $H_i:=F(\gamma_i)$ for $\gamma_i\in O_i$ ($i=1,2,3$).
  Recall that $-H_i=H_i$ for $i=1,2,3$ and $0\in H_1,H_3$.

\subsubsection{Case $n=1$}
One has $BC_1=\{\pm\vareps_1,\pm 2\vareps_1\}$, $X:=\{\vareps_1\}$.
 (GR2) is equivalent to
\begin{equation}\label{eqBC1}
0\in H_1,\ \ \ H_1+2H_1, H_1+H_2\subset H_1,\ \ \ H_2+2H_2, H_2+4H_1\subset H_2.
\end{equation}
Therefore $L:=\mathbb{Z}R\cap Ker(-,-)$ is spanned by $H_1$ and
$$H_1+2L\subset H_1,\ H_2+4L\subset H_2,\ H_2\subset H_1,\ H_2+2H_2\subset H_2.$$
As in~\Thm{thmAGRS1}, we introduce the canonical map
$\iota_2: L\to L/2L\cong\mathbb{F}_2^k$ (where
$k:=\dim Ker(-,-)$). Using~\Thm{thmAGRS1} (ii) we conclude that
$$R=\{\pm (\vareps_1+ \iota_2^{-1}(S)\}\cup\{\pm (2\vareps_1+ H_2)\},$$
where $S\subset \mathbb{F}_2^k=L/2L$ contains zero and a basis
of $\mathbb{F}_2^k$ and $H_2\subset \iota_2^{-1}(S)$ satisfying
$$H_2+4L, H_2+2H_2\subset H_2.$$

\subsubsection{Case $n\geq 2$}
(GR2) is   equivalent to~(\ref{eqBC1}) and the following
conditions on $H_3$:
$$0\in H_3,\ \ H_1+H_3\subset H_1,\ \ H_2+2H_3\subset H_2,\
\ H_3+2H_1, H_3+H_2, H_3+2H_3\subset H_3,$$
and $H_3+H_3\subset H_3$ for $n>2$.
Set
$$L:=\mathbb{Z}H_3;$$
by above, $Ker(-,-)=L\otimes_{\mathbb{Z}}\mathbb{C}$, so
$L$ has rank $k$.

For $n>2$ each $R$ is of the form $R(S_1,S_2)$, where
$S_1,S_2\subset \mathbb{F}_2^k$ and $0\in S_1$ and $R(S_1,S_2)$
can be described as follows:

$H_1\subset \frac{1}{2}L$ is the preimage of $S_1$ in $\frac{1}{2}L\to \mathbb{F}_2^k=\frac{1}{2}L/L$;

$H_2\subset L$ is the preimage of $S_2$ in $L\to \mathbb{F}_2^k=L/2L$; $H_3=L$.

For $n=2$ each $R$ is of the form
$R(S_1,S_2,H_3)$, where  $S_1,S_2$ as above
($S_1,S_2\subset \mathbb{F}_2^k$ and $0\in S_1$),
and $H_3\subset L$ contains $0$, a basis of $L$ and satisfies
$H_2+H_3=2H_1+H_3\subset H_3$ (where $H_1,H_2$ are as for $n>2$).

\subsection{}
\begin{prop}{caseBCmn}
Let $R\subset V$ be a GRRS with $k:=\dim Ker(-,-)$.

(i) The isomorphism classes of GRRSs $R$ with $cl(R)=C(m,n), mn>1$
are in one-to-one correspondence with the equivalence classes
of the proper non-empty
subsets $S$ of  the affine space  $\mathbb{F}_2^k$
up to  to the action of an affine automorphism of $\mathbb{F}_2^k$, see~\S~\ref{BCRS} for the description of $R(S)$.
For $m=n$ there is an additional isomorphism
$R(S)\cong R(\mathbb{F}_2^k\setminus S)$.

(ii) The isomorphism classes of GRRSs $R$ with $cl(R)=BC(m,n)$
are in one-to-one correspondence with the equivalence classes
of the pairs of a proper non-empty subset $S$
and a non-empty subset $S'$ of  the affine space  $\mathbb{F}_2^k$
up to the action of an affine automorphism of $\mathbb{F}_2^k$,
 see~\S~\ref{BCRS} for the description of $R(S,S')$.
 For $m=n$ there is an additional isomorphism
$R(S, S')\cong R(\mathbb{F}_2^k\setminus S, S')$.

(iii) If $R$ is a GRRS such that $cl(R)=C(1,1)$, then either $R\cong A(1,1)^{(k-1)}$ or $R$ is a "rational quotient" $A(1,1)^{(k)}_x$ (for $k=1$ one has
$x\in \mathbb{Q}, 0<x<1/2$, and
for $k>1$ one has $x=1/q$, where $q\in\mathbb{Z}_{>0}$)
of $A(1,1)^{(k)}$, or
$R\cong C(1,1)(S)$ for some non-empty $S\subset\mathbb{F}_2^k$,
see~\S~\ref{BCRS}.
The only isomorphic GRRSs are
$C(1,1)(S)\cong C(1,1)(S')$, where
$S'=\psi(S)$, where $\psi: \mathbb{F}_2^k\to\mathbb{F}_2^k$
is an affine automorphism and
$R(S)\cong R(\mathbb{F}_2^k\setminus S)$.
\end{prop}

\subsubsection{Description of $R(S)$}\label{BCRS}
In order to describe the above  correspondences in (i)--(iii)
between GRRSs and subsets
in $\mathbb{F}_2^k$ we fix a  free abelian group
 $L\subset Ker (-,-)$  of rank $k$ and
 denote by $\iota_2$  the canonical map
$\iota_2: L\to L/2L\cong\mathbb{F}_2^k$ and by $\iota^{-1}_2$
the preimage of $S\subset \mathbb{F}_2^k$ in $L$.

If $S$ contains zero, then for $cl(R)=C(m,n)$ we take

$$R(S):=\{\pm\vareps_i\pm\vareps_j+L; \pm\delta_s\pm\delta_t+L;
\pm \vareps_i\pm\delta_j+L;
\pm 2\vareps_i+ \iota_2^{-1}(S); \pm 2\delta_s+(L\setminus\iota_2^{-1}(S))\}_{{1\leq i\not=j\leq m}
\atop{1\leq s\not=t\leq n}}.$$

For $BC(m,n)$ we construct $R(S,S')$ by adding to $R(S)$ the roots
$$\{\pm \vareps_i+\frac{1}{2}\iota_2^{-1}(S'); \pm \delta_s+
\frac{1}{2} \iota_2^{-1}(S')\}_{1\leq i\not=j\leq m, 1\leq s\not=t\leq n}.$$

For an arbitrary subset $S$, we take $R(S):=R(S-s)$ (resp., $R(S,S'):=R(S-s,S')$) for some $s\in S$
 the result does not depend on the choice of $s\in S$,
see~\S~\ref{choices}.

\subsubsection{Case $\dim Ker (-,-)=1$}
In this case~\Prop{caseBCmn} gives the following:

for $cl(R)=C(1,1)$, $R$ is either a finite GRRS $A(1,1)$
or $A(1,1)_x$ for $x\in\mathbb{Q}$, or $R(0)$ ($\cong A(1,1)_{1/2}$);

for $cl(R)=C(m,n)$, $R$ is $R(0)$ ($\cong A(2m-1,2n)^{(2)}$);

for $cl(R)=BC(m,n)$, $R$ is either $R(0,0)$ ($\cong A(2n,2m-1)^{(2)}$), or $R(0,1)$ ($\cong A(2m,2n-1)^{(2)}$),
or $R(0,\mathbb{F}_2)$ ($\cong A(2m,2n)^{(4)}$).
Note that $R(0,0)\cong R(1,0)$, $R(1,1)\cong R(0,1)$
and all these GRRSs are isomorphic if $m=n$.

\subsubsection{Isomorphisms}\label{Cmniso}
The conditions when $R(S), R(S')$ (resp., $R(S,S')$ and $R(S_1, S_1')$) are isomorphic can be proven similarly to~\S~\ref{isoRS}. For $m=n$ the involution
$\vareps_i\mapsto\delta_i$ gives rise to the isomorphism
$R(S)\cong R(\mathbb{F}_2^k\setminus S)$
(resp., $R(S, S')\cong R(\mathbb{F}_2^k\setminus S, S')$).

Remark that for $cl(R)=C(1,1)$ one has $A(1,1)_{1/2}\cong R(0)$.
However, in~\Prop{caseBCmn}
we consider only $A(1,1)_x$ for $0<x<1/2$, so
this isomorphism is not mentioned.

\subsection{Proof of~\Prop{caseBCmn}}
Let $X$ be a set of simple roots of $C_m\coprod C_n\subset C(m,n)\subset
 BC(m,n)$ (i.e., $X=:\{\vareps_1-\vareps_2,\ldots,\vareps_{m-1}-\vareps_m,
 2\vareps_m,\delta_1-\delta_2,\ldots,2\delta_n\}$);
 applying the procedure described in~\S~\ref{Falphastr},
 we get $R'=C_m\coprod C_n$. The $W(R')$ orbits in $cl(R)$ are
 the following: the set of isotropic roots, the set of long
 roots of $C_m$ (resp., of $C_n$),
  the set of short roots of $C_m$ (resp., of $C_n$), and
   for $BC(m,n)$, the set of short roots of $B_m$ (resp., of $B_n$).
   Recall that $F(\alpha)$ is the same for elements in the same orbit.

 Since  all isotropic
 roots form one $W(R')$-orbit, $F(-\alpha)=F(\alpha)$
 for each isotropic $\alpha$; since $F(-\alpha)=-F(\alpha)$, we get $F(\alpha)=-F(\alpha)$.

 We claim that
 \begin{equation}\label{MM}
 \begin{array}{l}
 \forall x,y\in F(\vareps_1-\delta_1)\ \text{ exactly one holds }
 x+y\in F(2\vareps_1)\ \text{ or } x-y\in F(2\delta_1),\\
 F(\vareps_1-\delta_1)+F(2\vareps_1),F(\vareps_1-\delta_1)+F(2\delta_1)
 \subset
 F(\vareps_1-\delta_1)
 \end{array}
 \end{equation}
 Indeed, for each $x,y\in F(\vareps_1-\delta_1)$
 one has $\vareps_1-\delta_1+x,
\vareps_1+\delta_1+y\in R$ so exactly one of two elements $2\vareps_1+x+y$ and
$2\delta_1+y-x$ lies in $R$ (see~\S~\ref{alphabeta}). This establishes the first formula.
The other formulae follow from~(\ref{eqFalphabeta}).

Set
$$L':=\mathbb{Z}(F(2\vareps_1)\cup F(2\delta_1)).$$
Take any $a\in F(\vareps_1-\delta_1)$. By~(\ref{MM}),
\begin{equation}\label{Fe1d1}
F(\vareps_1-\delta_1)=L'\pm a
\end{equation}
and, moreover, for each $b\in L'$ exactly one holds:
$b\in F(2\vareps_1)$ or $b+2a\in F(2\delta_1)$, and, similarly,
$b+2a\in F(2\vareps_1)$ or $b\in F(2\delta_1)$.
Therefore
\begin{equation}\label{C11eq}
L'=F(2\vareps_1)\coprod (F(2\delta_1)-2a)\cap L')=F(2\delta_1)\coprod (F(2\vareps_1)-2a)\cap L').
\end{equation}
Note that $0\in F(2\vareps_1), F(2\delta_1)$ gives $a\not\in L'$.

\subsubsection{Case $C(1,1)$}\label{C11}
If $2a\not\in L'$, then
$F(2\vareps_1)=F(2\delta_1)=L'$. If $a\not\in \mathbb{Q}L'$, then
$R\cong A(1,1)^{(k-1)}$, where $k=\dim \mathbb{Z}R\cap (-,-)$
(if $k=1$, then $L'=0$ and $R=A(1,1)$), otherwise
$R\cong A(1,1)_x^{(k-1)}$, see the proof of~\Prop{propAnnx}.
Notice that for $x=p/q$, $2q a\in L'$; we exclude
$q=2$, since $A(1,1)_{1/2}\cong R(0)$, see~\ref{BC11} below.

Consider the case $2a\in L'$.
Since~(\ref{C11eq}) holds for each $a\in F(\vareps_1-\delta_1)$, one has $F(2\delta_1)+2L'=F(2\delta_1)$
and $F(2\vareps_1)+2L'=F(2\vareps_1)$. Now taking
$S:=\iota_2(F(2\vareps_1))$ and the automorphism $\delta_i\mapsto
\delta_i-a$ we get $R\cong R(S)$ as required.

\subsubsection{Case $C(m,n)$ with $mn>1$}\label{Cmn3}
Since $C(m,n)\cong C(n,m)$ we can (and will) assume that $m\geq 2$. Using~(\ref{Fe1d1}) we get
$$F(\vareps_1-\vareps_2)=
F(\vareps_1-\delta_1)+F(\vareps_1-\delta_1)=L'\pm 2a.$$
Since $0\in F(\vareps_1-\vareps_2)$ we obtain
$2a\in L'$ and thus $R\cong R(S)$.

\subsubsection{Case $BC(m,n)$}\label{BC11}
The additional relations include
$$\begin{array}{l}
F(\vareps_1-\delta_1)+F(\delta_1)=F(\vareps_1),\ \ \ \
F(\vareps_1-\delta_1)+F(\vareps_1)=F(\delta_1),\\
F(\vareps_1-\delta_1)+2F(\delta_1),
F(\vareps_1-\delta_1)+2F(\vareps_1)\subset F(\vareps_1-\delta_1),\\
F(\vareps_1)+F(2\vareps_1)\subset F(\vareps_1),
4F(\vareps_1)+F(2\vareps_1)\subset F(2\vareps_1),
\end{array}$$
and similar relations between $F(\delta_1)$ and $F(2\delta_1)$.
In particular,
$$F(\vareps_1-\delta_1)+F(\vareps_1-\delta_1)+F(\vareps_1)=
F(\vareps_1)$$
(so $L'+F(\vareps_1)=F(\vareps_1)$),
and, since $F(\vareps_1-\delta_1)=L'\pm a$,
$2F(\vareps_1)\subset L'\cup (L'-2a)$. Moreover,
$4F(\vareps_1)\subset L'$.

Take $b\in F(\vareps_1)$
and observe that $b\pm 2a\in F(\vareps_1)$.
Since $2F(\vareps_1)\subset L'\cup (L'-2a)$
we get $4a\in L'$ (and $a\not\in L'$ by (\ref{C11eq})).

If $2a\in L'$, we obtain $2F(\vareps_1)\subset L'$
and taking
$S:=\iota_2(F(2\vareps_1))$,  $S':=\iota_2(2F(\vareps_1))$
and the automorphism $\delta_i\mapsto
\delta_i-a$ we get $R\cong R(S, S')$ as required.

Let $2a\not\in L'$ (and $2a\in \frac{1}{2}L'$).
Consider an automorphism $\psi: V\to V$ which maps $\delta_1$ to
$\delta_1+a$, and stabilizes
$\vareps_1$ and the elements of $Ker(-,-)$.
Note that $L'$ constructed for $\psi(R)$ is $L'\cup (L'+2a)$,
which is a free group of the same rank as $L'$; moreover,
$$F(\psi(\vareps_1+\delta_1))=L'\cup (L'+2a),\ \
F(\psi(2\vareps_1))L',\ \ F(\psi(2\delta_1))=(L'+2a),$$
and so $\psi(R)=R(\mathbb{F}_2^{k-1},S')$, see above.
This completes the proof of~\Prop{caseBCmn}.

\section{GRRS with finite $cl(R)$ and $\dim Ker(-,-)=1$}\label{sect7}
From  the above results, it follows that
the only finite GRRS with a degenerate form $(-,-)$ is $A(n,n)$
(the root system of $\mathfrak{gl}(n,n)$).
As a consequence, if $cl(R)$ is finite and
$R\not=A(n,n)$, then $R$ is affine.

Symmetrizable affine Kac-Moody superalgebras were classified
in~\cite{K2},\cite{vdL}.
Summarizing the above results in the special case when $R$ is
an affine GRRS and
$\dim Ker(-,-)=1$, we see that such GRRSs correspond to
the real roots of symmetrizable affine Kac-Moody superalgebras.
More precisely, except for the case when
$cl(R)$  is the root system of $\mathfrak{psl}(n,n)$,
$n\geq 2$, $R$ is the set of real roots
of some affine Kac-Moody superalgebra $\fg$, see below.
If $cl(R)$  is the root system of $\mathfrak{psl}(n,n)$,
$n\geq 2$, $R$ is a quotient of the set of real roots
of $\mathfrak{pgl}(n,n)^{(1)}$. Conversely: the set of real roots of any affine Kac-Moody superalgebra other than $\mathfrak{gl}(n,n)^{(1)}$ is an affine GRRS with $\dim Ker(-,-)=1$.

If $cl(R)$ is one of $A_n,D_n, E_6,E_7,E_8, A(m,n), m\not=n, C(n), D(m,n), D(2,1,a), F(4), G(3)$,
then  $\fg$ is  the corresponding non-twisted affine Kac-Moody superalgebra ($R=cl(R)^{(1)}$).

If $cl(R)$ is one of the GRRSs $B_n,C_n, F_4,G_2$ and $B(m,n)$ with $m,n\geq 1$,
then $\fg$ is either  the corresponding non-twisted affine Kac-Moody superalgebra or the twisted affine Lie superalgebra
$D_{n+1}^{(2)}, A_{2n-1}^{(2)}, E_6^{(2)}, D_4^{(3)}$ and $D(m+1,n)^{(2)}$ respectively.

If  $cl(R)$ is the non-reduced root system
 $BC_n=B(0,n)$ ($n\geq 1$), then $\fg$ can be
 $B(0,n)^{(1)}, A_{2n}^{(2)}, A(0,2n-1)^{(2)}, C(n+1)^{(2)}$ or $A(0,2n)^{(4)}$ (where $A(0,1)^{(2)}\cong C(2)^{(2)}$
 as $A(0,1)\cong C(2)$).

If  $cl(R)=BC(m,n)$ ($m,n\geq 1$), then
$\fg=A(2m,2n-1)^{(2)},\ A(2n,2m-1)^{(2)}$ or $A(2m,2n)^{(4)}$.
If  $cl(R)=C(m,n)$ with $mn>1$, then
$\fg=A(2m-1,2n-1)^{(2)}$.

\end{document}